\documentclass[11pt]{amsart}

\title{}
\author{}
\date{\today}

\usepackage{amsmath,amsfonts,amssymb,amsthm,graphicx,overpic,hyperref}
\usepackage{float,enumitem}
\usepackage{color,xcolor}
\usepackage[colorinlistoftodos]{todonotes}
\usepackage{multirow}
\usepackage{pgfplots}
\usepackage{subcaption}
\pgfplotsset{compat=1.7}

\usepackage[margin=2.7cm]{geometry}  
\numberwithin{equation}{section}

\newtheorem{thm}{Theorem}[section]
\newtheorem{prop}[thm]{Proposition}
\newtheorem{prp}[thm]{Proposition}
\newtheorem{cor}[thm]{Corollary}
\newtheorem{lem}[thm]{Lemma}

\theoremstyle{definition}

\newtheorem{que}{Question}

\newcommand{\nc}{\newcommand}
\nc{\dmo}{\DeclareMathOperator}
\usepackage{dsfont}

\nc{\abs}[1]{\left| #1 \right|}
\nc{\bigO}[1]{O\left(#1\right)}
\nc{\card}[1]{\left|#1\right|}
\nc{\ceil}[1]{\left\lceil #1 \right\rceil}
\nc{\CC}{\mathbb{C}}
\nc{\dilog}{\mathcal{L}}
\nc{\floor}[1]{\left\lfloor #1 \right\rfloor}
\nc{\ind}{\mathds{1}}
\nc{\ZZ}{\mathbb{Z}}
\nc{\len}[1]{\left| #1 \right|}
\nc{\littleo}[1]{o\left(#1\right)}
\dmo{\Mat}{Mat}
\nc{\NN}{\mathbb{N}}
\nc{\norm}[1]{\left|\left| #1 \right|\right|}
\nc{\QQ}{\mathbb{Q}}
\nc{\RR}{\mathbb{R}}
\nc{\st}[2]{\left\{\, #1 \,:\, #2\,\right\}}
\dmo{\supp}{supp}
\nc{\tr}[1]{\mathrm{tr}\left(#1\right)}
\nc{\what}{\widehat}
\dmo{\im}{Im}
\nc{\eps}{\varepsilon}
\dmo{\li}{li}
\dmo{\arccosh}{arccosh}

\dmo{\area}{area}
\dmo{\conv}{conv}
\dmo{\diam}{diam}
\dmo{\DD}{\mathbb{D}}
\dmo{\dist}{\mathrm{d}}
\nc{\HH}{\mathbb{H}}
\dmo{\Isom}{Isom}
\dmo{\MCG}{MCG}
\dmo{\MPL}{MPL}
\dmo{\Mod}{\mathcal{M}}
\dmo{\PL}{PL}
\nc{\Sphere}{\mathbb{S}}
\dmo{\sys}{sys}
\dmo{\kiss}{Kiss}
\dmo{\Teich}{\mathcal{T}}
\nc{\Torus}{\mathbb{T}}
\dmo{\vol}{vol}
\dmo{\WP}{WP}

\dmo{\convTV}{\;\stackrel{\mathrm{TV}}{\longrightarrow}\;}
\nc{\ExV}[2]{\mathbb{E}_{#1}\left[#2\right]}
\dmo{\EE}{\mathbb{E}}
\nc{\Pro}[2]{\mathbb{P}_{#1}\left[#2\right]}
\dmo{\PP}{\mathbb{P}}
\nc{\distTV}[2]{\mathrm{d}_{\rm TV}\left(#1,#2\right)}
\dmo{\UU}{\mathbb{U}}
\nc{\Var}[2]{\mathbb{V}\mathrm{ar}_{#1}\left[#2\right]}

\dmo{\alt}{\mathfrak{A}}
\dmo{\Aut}{Aut}
\dmo{\Fix}{Fix}
\dmo{\GL}{GL}
\dmo{\Hom}{Hom}
\dmo{\id}{Id}
\dmo{\PGL}{PGL}
\dmo{\PSL}{PSL}
\dmo{\PO}{PO}
\dmo{\Rep}{Rep}
\dmo{\SL}{SL}
\dmo{\SO}{SO}
\dmo{\sym}{\mathfrak{S}}
\dmo{\inv}{\mathcal{I}}
\dmo{\orb}{\mathcal{O}}
\dmo{\stab}{Stab}

\dmo{\calA}{\mathcal{A}}
\dmo{\calB}{\mathcal{B}}
\dmo{\calC}{\mathcal{C}}
\dmo{\calD}{\mathcal{D}}
\dmo{\calE}{\mathcal{E}}
\dmo{\calF}{\mathcal{F}}
\dmo{\calG}{\mathcal{G}}
\dmo{\calH}{\mathcal{H}}
\dmo{\calI}{\mathcal{I}}
\dmo{\calJ}{\mathcal{J}}
\dmo{\calK}{\mathcal{K}}
\dmo{\calL}{\mathcal{L}}
\dmo{\calM}{\mathcal{M}}
\dmo{\calN}{\mathcal{N}}
\dmo{\calO}{\mathcal{O}}
\dmo{\calP}{\mathcal{P}}
\dmo{\calQ}{\mathcal{Q}}
\dmo{\calR}{\mathcal{R}}
\dmo{\calS}{\mathcal{S}}
\dmo{\calT}{\mathcal{T}}
\dmo{\calU}{\mathcal{U}}
\dmo{\calV}{\mathcal{V}}
\dmo{\calW}{\mathcal{W}}
\dmo{\calX}{\mathcal{X}}
\dmo{\calY}{\mathcal{Y}}
\dmo{\calZ}{\mathcal{Z}}

\nc{\klav}{Klav\v{z}ar}

\nc{\bi}{\mathbf{i}}
\nc{\bj}{\mathbf{j}}
\nc{\bk}{\mathbf{k}}

\newcommand{\interval}[4]{
  \ifthenelse{ \equal{#1}{o} } {\mathopen{]}} {\mathopen{[}}
  #2, #3
  \ifthenelse{ \equal{#4}{o} } {\mathclose{[}} {\mathclose{]}}
}

\newcommand{\ovl}[1]{\overline #1}

\newcommand{\pl}{\partial}

\newcommand{\inj}{\operatorname{inj}}

\newcommand{\double}{\mathcal{D}}

\title{A model for random three--manifolds}

\author{Bram Petri and Jean Raimbault}
  \address{Institut de Math\'ematiques de Jussieu--Paris Rive Gauche ; UMR7586 \\ Sorbonne Universit\'e - Campus Pierre et Marie Curie \\
4, place Jussieu, 75252 Paris Cedex 05, France}
  \email{Bram.Petri@imj-prg.fr}
  \address{Institut de Math\'ematiques de Toulouse ; UMR5219 \\ Universit\'e de Toulouse ; CNRS \\ UPS IMT, F-31062 Toulouse Cedex 9, France}
  \email{Jean.Raimbault@math.univ-toulouse.fr}

\begin{document}

\begin{abstract}
We study compact three-manifolds with boundary obtained by randomly gluing together truncated tetrahedra along their faces. We prove that, asymptotically almost surely as the number of tetrahedra tends to infinity, these manifolds are connected and have a single boundary component. We prove a law of large numbers for the genus of this boundary component, we show that the Heegaard genus of these manifolds is linear in the number of tetrahedra and we bound their first Betti number.

We also show that, asymptotically almost surely as the number of tetrahedra tends to infinity, our manifolds admit a unique hyperbolic metric with totally geodesic boundary. We prove a law of large numbers for the volume of this metric, prove that the associated Laplacian has a uniform spectral gap and show that the diameter of our manifolds is logarithmic as a function of their volume. Finally, we determine the Benjamini--Schramm limit of our sequence of random manifolds.
\end{abstract}

\maketitle

\section{Introduction}

\subsection{Context}

Random constructions of compact manifolds can be seen as an analogue of the well-established theory of random graphs and serve similar purposes. First of all, they make the notion of a ``typical'' manifold rigorous. Secondly, they can be used as a testing ground for conjectures of which the proof is still out of reach. Finally, there is what is often called the \emph{probabilistic method} -- using probability theory to prove the existence of objects with extremal properties. In this paper we are mostly interested in the first aspect. 

Let us be more specific about what kind of objects we are intetested in. As is the case for graphs, there are countably many homeomorphism types of compact manifolds. Thus a random manifold consists not in one random variable but rather a family of random variables---say $M_n, n \ge 1$---where $n$ is some measure of ``complexity'', usually in relation with a particular construction that is used to define the random objects. For graphs this will often be the number of vertices. Random models for 3--manifolds that have been well-studied are random Heegaard splittings and random fibered manifolds; here the complexity depends on two integers: the genus $g$ of the handlebody or the fiber, and the number of steps $k$ used to generate a random mapping class \cite[Section 2.10]{DT}. A basic property should be that the union of the support of the $M_n$ is the whole set of the manifolds one is interested in studying. The models above satisfy this requirement if one takes both $k, g \to +\infty$ (though only virtually for the second one) but the asymptotic results pertaining to them (in particular hyperbolicity) are mostly studied in terms of the mapping class (that is, when $k \to +\infty$). If one is interested in studying typical 3--manifolds this does not seem satisfying. 

A more direct measure of the complexity of a 3--manifold is the minimal number of tetrahedra needed to triangulate it. 
A natural way to construct random 3--manifolds is thus to start with a model for a random triangulation on $n$ tetrahedra and condition it to be a manifold. However, studying such a model of a random manifold is hard because if one randomly glues the faces of $n$ tetrahedra together in pairs, the probability that the result is a manifold tends to $0$ as $n\to\infty$ (see for instance \cite[Proposition 2.8]{DT}). So we cannot rely on the study of a generic triangulation to establish a.a.s. properties of the manifolds and we have to instead study probabilities conditioned on a set of conditions that is hard to manage.

We will not adress this issue in this paper, but note that even counting the number of triangulations is a hard problem (the best known bounds we are aware of are due to Chapuy--Perarnau \cite{CP}). Instead we will consider compact manifolds with boundary associated with random triangulations. The only points in a 3-dimensional triangulation that might not admit neighbourhoods homeomorphic to open sets in $\RR^3$ are the vertices. As such, we obtain a random 3--manifold with boundary by randomly gluing together $n$ tetrahedra that are truncated near the vertices (see Figure \ref{pic_truncatedtet}). Moreover all compact 3--manifolds with non-empty boundary can be obtained in this way (see for example \cite[Corollary 1.3]{Costantino_Frigerio_Petronio}).

\begin{figure}[!h]
\begin{center}
\includegraphics[scale=1]{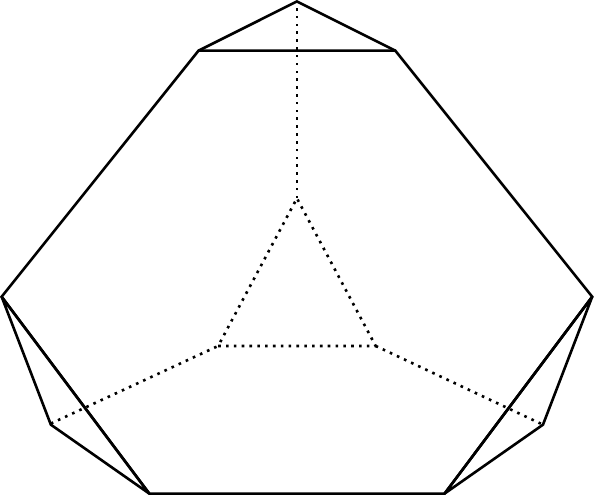}
\caption{A truncated tetrahedron. $M_n$ is built by randomly gluing $n$ copies of this polyhedron together along their hexagonal faces.}\label{pic_truncatedtet}
\end{center}
\end{figure}

We are interested in the asymptotic behaviour as $n \to +\infty$ of geometric and topological properties of $M_n$. We are particularly interested in finding properties whose probability of occurence is asymptotically 1 (for regular graphs this can for instance be connectivity or expansion, depending on the model). For 3--manifolds the most obvious candidate for such a property is hyperbolicity. We will prove that a.a.s. our manifolds are hyperbolic (with totally geodesic boundary) with volume proportional to the number of tetrahedra, their Heegaard genus goes to infinity, and get estimates on their Betti numbers. We also prove some finer results about their geometry: they are expanders and we show the converge to an explicit limit in a probabilistic version of the Gromov--Hausdorff topology. 


\subsection{Results}
We will impose one further condition: we condition on two tetrahedra sharing at most one face and every face being incident to two distinct tetrahedra. This is strictly weaker than asking that the complex be simplicial\footnote{Even if it can be argued that this is not a very unnatural constraint, our main reason for setting this constraint is a technical one: we need it in the proof of Lemma \ref{small_cusps}}. The resulting random manifold will be called $M_n$. A detailed description of the model can be found in Section \ref{delta-complex}.

The first question now is what the topology of the resulting manifold is. We prove:
\begin{thm}[Topology]\label{thm_main_topology}
\begin{itemize}
\item[(a)] We have
\[ \lim_{n\to\infty} \PP[M_n \text{ is connected and has a single boundary component}] = 1\]
\item[(b)] The genus $g(\partial M_n)$ of the boundary of $M_n$ satisfies
\[ g(\partial M_n) \sim n \quad \text{as } n\to\infty\]
in probability.
\item[(c)] Let $\double M_n$ denotes the double of $M_n$ along its boundary and $g(\double M_n)$ its Heegard genus. Then
  \[
  \lim_{n\to+\infty} \mathbb P[n - \theta(n) \le g(\double M_n) \le n + \theta(n)] = 1,
  \]
  for any function $\theta:\NN\to\RR$ that grows super-logarithmically\footnote{By this we mean that $\lim_{n\to\infty}\frac{\theta(n)}{\log(n)}=+\infty$.}.
\item[(d)] There exists $C$ such that the Betti numbers $b_1(M_n)$ and $b_1(M_n, \partial M_n)$ satisfy
  \[
  \lim_{n\to+\infty} \mathbb P[b_1(M_n, \partial M_n) \le \theta(n)] = 1,\, \lim_{n\to+\infty} \mathbb P[\;|b_1(M_n) - n| \le \theta(n)] = 1
  \]  
  for any function $\theta:\NN\to\RR$ that grows super-logarithmically.
\end{itemize}
\end{thm}

This is a combination of Corollary \ref{cor_topology}, \ref{betti_comb} and \ref{double_genus}. Moreover, in Theorem \ref{thm_edges} below we prove various combinatorial properties of the interior edges in our random complex. In item (c) we look at the Heegaard genus of the double rather than the usual notion of Heegard genus of the manifold itself (defined in terms of decompositions with compression bodies, cf. \cite[2.2]{Scharelmann_handbook}) because the latter is bounded below by the genus of the boundary, so (c) says something that is not covered by (b). 

In low dimensions it turns out that typical objects are often hyperbolic and in that sense, our model is no different. Note that it follows from Mostow rigidity that if $M_n$ caries a hyperbolic metric with totally geodedic boundary, then this metric is unique up to isometry. As such, one can also ask for the geometric properties of this metric. We prove:

\begin{thm}[Geometry]~\label{thm_main_geometry}
We have
\[ \lim_{n\to +\infty} \PP[M_n \text{ carries a hyperbolic metric with totally geodesic boundary}] = 1.\]
This metric has the following properties:
\begin{itemize}
\item[(a)]  The hyperbolic volume $\vol(M_n)$ of $M_n$ satisfies:
\[ \vol(M_n) \sim n\cdot v_O \quad \text{as } n\to\infty\]
in probability. Here $v_O$ denotes the volume of the regular right angled ideal hyperbolic octahedron.
\item[(b)] There exists a constant $c_\lambda>0$ so that the first discrete Laplacian eigenvalue $\lambda_1(M_n)$ of $M_n$ satisfies
\[ \lim_{n\to +\infty} \PP[\lambda_1(M_n) > c_\lambda] = 1. \]
\item[(c)] There exists a constant $c_d>0$ such that the diameter $\diam(M_n)$ of $M_n$ satisfies:
\[
\lim_{n\to +\infty}\PP[\diam(M_n) < c_d \log(\vol(M_n))]  = 1
\]
\item[(d)] There exists a constant $c_s>0$ such that the systole\footnote{The {\em systole} of a compact manifold is the smallest length of a closed geodesic; we do not take it to include lengths of arcs with endpoints on the boundary ; see next item for this. }  $\sys(M_n)$ of $M_n$ satisfies:
\[
\lim_{n\to +\infty}\PP[\sys(M_n) > c_s ]  = 1
\]
\item[(e)] For every $\eps>0$,
\[
\lim_{n\to +\infty}\PP\left[\sys(\double M_n) <  \frac{1}{n^{1/4-\eps}} \right]  = 1.
\]
The same holds for the minimal length among arcs in $M_n$ that are homotopically non-trivial relative to $\partial M_n$.
\end{itemize}
\end{thm}

Some remarks about these results :
\begin{itemize}
\item Our proof of hyperbolisation for $M_n$ does not rely on Perelman's proof of the Geometrisation conjecture. Instead, we use Andreev's theorem \cite{Roeder_Hubbard} and recent work by Futer--Purcell--Schleimer \cite{Futer_Purcell_Schleimer} on Dehn fillings. Note that there is also a ``Ricci-flow-free'' proof of hyperbolisation of random Heegaard splittings \cite{FellerSistoViaggi}.

\item (b) admits a more geometric reformulation, as it follows from it together with classical work by Buser \cite{Buser} that the Cheeger constant of $M_n$ is also (asymptotically almost surely) uniformly bounded from below. It also implies (with (a) and a theorem of Lackenby \cite{Lackenby_Heegaard}) a weaker version of (c) in our topological theorem. 
  
\item (c) also implies that $\double M_n$ has logarithmic diameter. It also follows from an easy volume argument that the diameter of a closed hyperbolic $3$-manifold $M$ satisfies
\[
\diam(M) \geq \frac{1}{2} \log(\vol(M)) - C
\]
for some uniform constant $C>0$.

\item By arguments similar to those we use for (d) and (e) we could give probabilistic upper bounds for $\sys(M_n)$, and a lower bound for $\sys(\mathcal DM_n)$. The former are a bit awkward to state, and the latter would not be sharp. See also Question \ref{quest_sys} below. 
\end{itemize}

Besides expansion, another way of looking efficiently at the global geometry of a (possibly random) compact Riemannian manifold that has recently seen much interest is the so-called Benjamini--Schramm topology (see \cite{Gelander_lecture} for a survey). We determine the Benjamini--Schramm limit of the sequence $(M_n)_n$ as a consequence of our proof of hyperbolisation. This is more technical than our other results, so we will not give precise statements here but just a sketch of what this means. Very roughly, a sequence of finite volume random hyperbolic manifolds $(M_n)_n$ converges in the Benjamini--Schramm sense to a limit $M_\infty$ (a random pointed manifold) if for every fixed $R>0$, the $R$-neighbourhood of a uniformly random point in $M_n$ converges (in pointed Gromov--Hausdorff topology) to the $R$-neighbourhood of a random point in $M_\infty$. It turns out that the Benjamini--Schramm limit of $M_n$ can be identified with a tree of right angled octahedra pointed at a uniform random point (which makes sense since this manifold has a cofinite group of isometries). A rigorous exposition of these notions, and a precise statement for the result discussed above, are given in Section \ref{sec_BSconv}.

\subsection{Notes and references}

Various models for random manifolds are known in dimensions two (eg. \cite{BM,GPY,Mir,BCP,MageeNaudPuder}) and three (eg. \cite{DT}) and all three types of questions mentioned above have been explored: the models in \cite{BM} and \cite{Mir} are plausible as models of typical (hyperbolic) surfaces, the original motivation for the introduction of random Heegaard splittings by Dunfield and Thurston was to study the (at that point still unsolved) virtual Haken conjecture and \cite{GPY}, \cite{BCP} and \cite{LMW} are applications of the probabilistic method to produce hyperbolic surfaces without short pants decompositions, hyperbolic surfaces with near-minimal diameter and infinitely many closed hyperbolic homology three-spheres with a fixed Casson invariant and Heegaard genus respectively.

The most studied models for random three-manifolds are those of random Heegaard splittings and random mapping tori. Both of these are hyperbolic with probability $1$ \cite{Maher1,Maher2}. Moreover, like our manifolds (Theorem \ref{thm_main_geometry}(b)) they satisfy a law of large numbers for volume \cite{Viaggi}, with a constant depending on the underlying random walk. Their spectral gap behaves differently: it is inversely quadratic in volume \cite{HamenstaedtViaggi}. Their injectivity radius has been studied in \cite{SistoTaylor} and torsion in their homology in \cite{BBGHHKPV}. Moreover, even if random Heegaard splittings turn out to be hyperbolic with probability $1$ \cite{Maher1}, unlike for instance random regular graphs \cite{Bol1} and random hyperbolic surfaces \cite{BM,Mir,Raimbault}, they do not Benjamini--Schramm converge to their universal cover -- i.e. already at a bounded scale, the geometry of these manifolds ceases to be that of $\mathbb{H}^3$. 

Whether the model studied in this paper is plausible as a model for a ``typical'' three-manifold with boundary, we leave to the reader. We note that the question of studying this model has been evoked before (see for instance \cite[Question 6.2]{DelpHoffossManning}) but we aren't aware of any prior other results on it.

Finally, it is possible to derive models for random closed manifolds from our random manifolds with boundary. The simplest would be obtained by just doubling it; however this is supported only on manifolds admitting an involution with codimension-1 fixed locus. There are various ways to break this symmetry: the fact that the boundary is triangulated allows us to identify it to a fixed (depending only on the genus) model surface ``up to a finite ambiguity''. So we can talk about random mapping classes of the boundary almost as usual, and these allow us to perform various more complicated constructions such as gluing back a copy using a random mapping class. We can also glue the appropriate handlebody (or indeed any other manifold with connected boundary of the correct genus). By Geometrisation all these models are hyperbolic; however we do not know how their volume behaves, whether they are expanders (say for a choice of the mapping class with ``few steps'') or not, or whether they admit a Benjamini--Schramm limit. The investigation of such questions would certainly require a different set of tools than what we use here. 

\subsection{Proof ideas}
The order in which we prove our results is very different from the order in which we presented them above. The two big steps consist of understanding the combinatorial properties of the complex we build and then using those to understand the geometry and topology. 

The first observation is that all results above are of the form $\PP_n[P_n]\to 1$ as $n\to\infty$ for some sequence of properties $P_n$ of $M_n$. It follows from classical results in graph theory \cite{Bol1,Wor} that for such statements it is sufficient to prove the analogous statement for the random manifolds $N_n$ obtained by randomly gluing the building blocks together (without setting the condition on faces we set for $M_n$). 

In what follows, we will try to avoid repeating the phrase ``asymptotically almost surely'' and will often just say that $M_n$ has this or that property when we mean it has the given property asymptotically almost surely.

The proofs now start with the combinatorics (in Section \ref{sec_comb}). Using the observation above, the idea is to study the properties of $N_n$ and then turn these into properties of $M_n$. First of all, we prove, using elementary but tedious combinatorics, that the number of boundary components of $N_n$ (and hence $M_n$) is $1$. The next step, which is responsible for the largest part of the combinatorial arguments, is to study the combinatorics of interior edges in $N_n$. We ask two questions: how many edges are there? And, given some number $k\in\NN$, how many edges are there that are incident to $k$ truncated tetrahedra? To answer these questions, we will use \emph{peeling} techniques. These are techniques coming from the world of random planar maps (see for instance \cite{Curien}). The basic idea is to explore the random cell complex $N_n$ using a specific algorithm -- adapted to the problem at hand -- to determine in which order cells are explored. These lead to bounds on the expected number of interior edges (the total and the number that is incident to a fixed number of $3$-cells) that we think might be interesting in their own right (see Theorem \ref{thm_edges}). This in turn yields the Euler characteristic and hence the genus of the boundary of $N_n$ (and hence $M_n$). Note that all these combinatorial resuls can also be interpreted in terms of the cell complex obtained from gluing tetrahedra according to the same pattern -- a pseudo-manifold. 

After this, we deal with the geometric questions in Section \ref{sec_geom}. Some of our topological results also follow from these. Our first goal is hyperbolisation. The main idea behind our proof of this is to see our manifold as a Dehn filling of a non-compact hyperbolic three-manifold (similar ideas were used, with very different objectives, in \cite{Costantino_Frigerio_Petronio}). This non-compact manifold is obtained by gluing hyperbolic right-angled octahedra, using four alternating faces out of eight per octahedron, along the same pattern as $M_n$. We then first apply Andreev's theorem to fill cusps with ``few'' octahedra around them. The number such cusps is controlled by our combinatorial bounds. This creates another non-compact hyperbolic manifold, but without ``small'' cusps. We then fill the remaining cusps and rely on results by Futer--Purcell--Schleimer \cite{Futer_Purcell_Schleimer} to guarantee the result is hyperbolic. These same results also give us information about the way the geometry changes between the non-compact and the compact manifold. 

Once we have proved that $M_n$ is hyperbolic, we use results on random regular graphs together with a version of the Brooks--Burger transfer principle to show that $\lambda_1(M_n)$ can be uniformly bounded from below. Together with the law of larger number for volumes, which follows essentially directly from the geometric control we have over our hyperbolisations, and results by Lackenby we obtain the fact that the Heegaard genus of $\double M_n$ grows linearly in $n$. We prove the logarithmic bound on the diameter of $M_n$ by combining the fact that random $4$-regular graphs have logarithmic diameter with the geometric control we have over the change of geometry during Dehn filling.

Finally, we prove Benjamini--Schramm convergence, again using the geometric comparisson between the non-compact hyperbolic manifold and $M_n$. 

\subsection{Questions}

We finish this introduction with some questions.

\begin{que}[Poisson-Dirichlet distribution for edges]
Let us write 
\[ L= (L_1,L_2,\ldots ) \]
for the random vector that contains the \emph{lengths} of all the interior edges in $M_n$. Here the length of an edge is the number of $3$-cells incident to it and is counted with multiplicity -- i.e. if an interior edges is incident to a $3$-cell in multiple places then the $3$-cell is counted multiple times. 

If we order this vector so that $L_1\geq L_2 \geq \ldots$ and normalise it by dividing by the total length ($6n$), does the resulting partition of the interval $[0,1]$ converge in distribution to a Poisson-Dirichlet distributed variable? 

The analogous result is known to hold for surfaces obtained by randomly gluing polygons together \cite{Gamburd,ChmutovPittel,BCP2}.
%

\end{que}

\begin{que}[Explicit measures of expansion]
%

 Determine the optimal spectral gaps and Cheeger constants that hold a.a.s. for $M_n$. For instance, do we have
  \[
  \forall \eps > 0 \, \lim_{n \to +\infty} \mathbb P(\lambda_1(M_n) > 1 - \eps) = 1?
  \]
  An analogue of this is conjectured for random hyperbolic surfaces \cite[Problem 10.3]{Wright}, \cite[Conjecture 1.1]{MageeNaudPuder} and holds for random regular graphs \cite{Friedman}.
\end{que}

Finally we can also ask for sharp estimates for the systoles of $M_n$ and $\double M_n$.

\begin{que} \label{quest_sys}
  Give an explicit sequence $(s_n)$ such that $\sys(\mathcal DM_n) \sim s_n$ in probability. Compute (if it exists) $\lim \mathbb E(\sys(M_n))$. 
\end{que}


\subsubsection{Finer behaviour of homology and $L^2$-invariants of the limit}

In Theorem \ref{thm_main_topology}(d) we get good bounds for the typical Betti numbers of the $M_n$. However, in particular in view of the fact that the random Heegaard splittings of \cite{DT} typically have vanishing first Betti numbers, we ask the following question.

\begin{que} \label{betti_zero}
  Does $b_1(M_n, \partial M_n) = 0$ hold a.a.s.?
\end{que}

A positive answer is suggested by computer experiments conducted by Nathan Dunfield using Reigina. His results also suggest the following conjecture about the behaviour of the full integral homology, denoting by $H_1(M_n)_{\mathrm{tors}}$ the torsion subgroup  of $H_1(M_n)$: 

\begin{que} \label{tors_zero}
  Is $H_1(M_n)_{\mathrm{tors}}$ trivial a.a.s.? Or a weaker variant : do we have $\lim\limits_{n \to +\infty} \log \frac{|H_1(M_n)_{\mathrm{tors}}|}{n} = 0$ in probability? 
\end{que}

Note that for random Heegaard splittings the opposite behaviour occurs : $H_1(M_n)_{\mathrm{tors}}$ is it as large as possible, i.e. of exponential size in the number of thetrahedra (see \cite[Section 2.2]{Kowalski_bour}). 

In view of the convergence discussed in Section \ref{sec_BSconv} the last question could be related to the $L^2$-invariants of the infinite cover $O^\infty \to O$ (see \cite{Kammeyer} for an introduction to this topic). Our result on Betti numbers implies, via generalisations of the L\"uck Approximation Theorem (see \cite[5.4.3]{Kammeyer}) that the $L^2$-Betti numbers of $O^\infty \to O$ relative to the boundary vanish. We can ask about other $L^2$-invariants:

\begin{que}
  What are the Novikov--Shubin invariants of $O^\infty \to O$? Is its $L^2$-torsion equal to 0? 
\end{que}

In view of the approximation conjecture for torsion (which is wide open at present, see \cite[6.5]{Kammeyer} for a survey, but much simpler to deal with in 3-dimensions when the torsion vanishes, see \cite{Le_tors}), the vanishing of $L^2$-torsion would likely imply an affirmative answer to the weaker form of Question \ref{tors_zero}. We note that our expansion results implies (via the proof of L\"uck approximation) that the zeroth Novikov-Shubin invariant is $\infty^+$. 

\subsection*{Acknowledgements}
We worked on and off on this project for several years and as such it benefited from various grants. BP thanks the Max Planck Institute for Mathematics in Bonn and the ERC grant ``Moduli''. JR thanks the ANR for support through the projet ANR-16-CE40-0022-01 - AGIRA and the Hausdorff institute through the junior trimester ``Topologie'' which was held there in 2016. 

We also thank Fran\c{c}ois Costantino, Juan Souto and Gabriele Viaggi for useful discussions. BP thanks Thomas Budzinski and Nicolas Curien for teaching him how to peel a surface. We are indebted to Nathan Dunfield for various comments on a preliminary version, suggesting a more elementary and efficient approach to Heegaard genus and sharing the results of his simulation of the model with us.


\section{Combinatorics}\label{sec_comb}

In this section we formally describe the combinatorial model we use. Moreover, we determine the combinatorial structure of the random cell complex underlying our manifolds and derive some basic topological properties of our manifolds from it.

\subsection{The topological model} \label{delta-complex}

In what follows, $\calT_n$ will denote a $\Delta$-complex\footnote{A very mild generalization of a simplicial complex in which two $k$-faces are allowed to share more than one $(k-1)$-face and a $(k-1)$-face is allowed to be incident to a $k$ face on multiple sides.} obtained by randomly gluing the faces of $n$ tetrahedra together in pairs. The gluing that is used is picked at random among the three orientation reversing simplicial maps between the faces.

More formally, this goes as follows:
\begin{enumerate}
\item We start with $n$ \emph{labeled} tetrahedra. Here labeled means that every the vertex of these tetrahedra carries a unique label in $\{1,\ldots,4n\}$. If a face of a tetrahedron has vertices $v_1, v_2, v_3 \in \{1,\ldots,4n\}$, that moreover in this particular cyclic order induce an outward orientation on that particular face, then we will denote the face by the cycle\footnote{These cycles naturally lie in the symmetric group $\sym_{4n}$. However, because we won't use this in anything that follows, we will just think of these cycles as homeless.} $(v_1\;v_2\;v_3)$.
\item The faces are partitioned into $2n$ pairs, uniformly at random. We will denote the resulting partion by $\omega_n=(\omega_n^{(i)})_{i=1}^{2n}$.
\item Per pair of faces $\omega_n^{(i)}=\{(v_1\;v_2\;v_3),\;(w_1\;w_2\;w_3)\}$ in this partition, one of three cyclic-order-reversing pairings between the vertices is chosen uniformly at random. The resulting $2n$-tuple of pairings will be denoted $\sigma_n=(\sigma_n^{(i)})_{i=1}^{2n}$.
\item We identify each pair of faces $\omega_n^{(i)}=\{(v_1\;v_2\;v_3),\;(w_1\;w_2\;w_3)\}$ in the partition $\omega_n$ using the unique orientation reversing simplicial map that sends $v_j$ to $\sigma_n^{(i)}(v_j)$ for $j=1,2,3$. The resulting simplicial complex is called $\calT_n$.
\end{enumerate}

Let us write $(\Omega_n, \PP_n)$ for the corresponding probability space. So $\Omega_n$ is the finite set of all possibilities for $\omega_n$ and $\sigma_n$ and $\PP_n$ is the uniform probability measure on it. Note that
\[ \card{\Omega_n} = (4n)!! \; 3^{2n},\]
where for an even number $k\in \NN$, $k!!=(k-1)\cdot (k-3)\cdots 3 \cdot 1$.

The dual graph $\calG_n$ to $\calT_n$ - the $4$-valent graph whose vertices correspond to the tetrahedra of $\calT_n$ who share an edge per face that they have in common - is a random $4$-regular graph. The model this induces is exactly the \emph{configuration model}, one of the most studied models of random regular graphs (see eg. \cite{Bol1,Wor}).

Note that besides the number of tetrahedra ($n$), the number of $2$-faces is also deterministic in this model ($2n$). The numbers of vertices and edges are random variables.

\subsection{The results}

Let $N_n$ denote the manifold with boundary obtained by truncating $\calT_n$ at the vertices. Figure \ref{pic_truncatedtet} in the introduction shows what the basic building block of $N_n$ looks like.

In most of this section we will think in terms of $\calT_n$. However, since we are eventually interested in $N_n$, we will describe some of the results in terms of $N_n$.

\begin{thm}[Topology]\label{thm_topology}
\begin{itemize}
\item[(a)] We have
\[ \lim_{n\to\infty} \PP[N_n \text{ has a single boundary component}] = 1\]
\item[(b)] We have 
\[\EE[\chi(\partial N_n)]  = \log(n) - 2n +\bigO{1} \]
as $n\to\infty$. In particular, if we write $g(\partial N_n)$ for the genus of the single boundary component of $\calM_n$ we have that
\[ g(\partial N_n) \sim n \quad \text{as } n\to\infty\]
in probability.
\end{itemize}
\end{thm}

Part (a) follows from Proposition \ref{prp_vertices} and part (b) from Theorem \ref{thm_edges}. In the latter, we also prove bounds on the expected number of edges incident to a given number of tetrahedra in $\calT_n$. We will need these in the geometric part of the paper.

We will write $M_n$ for the random manifold we obtain if we condition on $\calG_n$ being simple, i.e. not having loops or multiple edges. Bollob\'as \cite{Bol2} proved that
\[\lim_{n\to\infty} \PP[\calG_n \text{ is simple }] > 0 .\]
In particular, this implies that if $(\calP_n)_n$ is a sequence of properties of $M_n$ and $N_n$ then as $n\to\infty$,
\begin{equation}\label{eq_asymptabscont}
 \text{if} \quad \PP[N_n \text{ has }\calP_n] \to 1 \quad \text{then} \quad  \PP[M_n \text{ has }\calP_n] \to 1.
\end{equation}

Combining this with the theorem above, we get:

\begin{cor} [Topology of the boundary] \label{cor_topology}
\begin{itemize}
\item[(a)] We have
\[ \lim_{n\to\infty} \PP[M_n \text{ has a single boundary component}] = 1\]
\item[(b)] We have 
\[ g(\partial M_n) \sim n \quad \text{as } n\to\infty\]
in probability.
\end{itemize}
\end{cor}

\subsection{The number of vertices}

We start by studying the number of vertices. Let us write $V$ for the number of vertices of $\calT_n$. Note that $V$ is also the number of boundary components of $N_n$. As such the following proposition implies Theorem \ref{thm_topology}(a).

\begin{prp}\label{prp_vertices} We have
\[\PP_n[V=1] \to 1, \]
as $n\to\infty$.
\end{prp}

\begin{proof} Let us write $V_{\mathrm{small}}:\Omega_n\to\NN$ for the random variable that counts the number of vertices of $\calT_n$ incident to most $2n$ tetrahedra (with multiplicity -- i.e. if a tetrahedron is incident to a vertex in multiple corners, it is counted multiple times). We will prove that
\[ \EE_n[V_{\mathrm{small}}]\to 0, \]
as $n\to\infty$. Note that this is sufficient to prove the proposition.

We will write 
\[\EE_n[V_{\mathrm{small}}] = \sum_a \EE_n[\ind_a].\]
Here the sum runs over ``labeled'' vertices $a$. A labelled vertex $a$ is the data of the gluings of all the labeled faces incident to a given vertex. Here $\ind_a:\Omega_n\to \{0,1\}$ is the indicator for the event that $a$ appears in $\omega$. As such
\[
\EE[\ind_a] = \frac{1}{3^{f} (4n-1)(4n-3) \cdots (4n-2f+1)},
\]
where $f$ is the number of faces incident to the given vertex.

Given $a$, write $n_1,n_2,n_3,n_4$ for the number of tetrahedra that are incident in $1$, $2$, $3$ and $4$ of their vertices to $a$ respectively. Note that the number of tetrahedron faces involved in such a gluing is given by
\[2f=3n_1+4(n_2+n_3+n_4) \]
This implies that $n_1$ must be even and that the number of gluings with the same selection of vertices is
\[3^{\frac{n_3}{2}+2n_4} (2f)!!.\]
The power of $3$ comes from the fact that only the faces with three vertices adjacent to them can be rotated. This number of faces is equal to $n_3+4n_4$.
As such we obtain
\[\EE_n[V_{\mathrm{small}}] = \sum_{\substack{0<n_1+n_2+n_3+n_4 \leq n  \\ n_1+2n_2+3n_3+4n_4 \leq 2n \\ n_1 \text{ even} }}
\binom{n}{n_1,n_2,n_3,n_4} \frac{3^{\frac{n_3}{2}+2n_4} (2f)!!}{3^{f} (4n-1)(4n-3)\cdots (4n-2f+1)}, \]
where we write
\[\binom{n}{n_1,n_2,n_3,n_4} = \frac{n!}{n_1!\cdot n_2!\cdot n_3!\cdot n_4!\cdot (n-n_1-n_2-n_3-n_4)!} \]
counts the number of subsets of the $n$ tetrahedra with the appropriate number of vertices in them.

We claim that this implies that $\EE_n[V_{\mathrm{small}}]=\bigO{n^{-1}}$. This follows by analyzing the terms in the sum. The largest terms in this sum are when $2f = 3n_1+4(n_2+n_3+n_4)$ is smallest. Indeed, using that $n_1+2n_2+3n_3+4n_4 \leq 2n$, an elementary but tedious computation shows that a term decreases when one of the $n_i's$ is increased. Given that the number of terms is quartic in $n$ and the number of terms that are larger than $\bigO{n^{-5}}$ is bounded, we obtain the estimate.
\end{proof}

\subsection{The number of edges}

The random variable that counts the number of edges that are incident to $k$ tetrahedra will be denoted by
\[ E_k: \Omega_n \to \NN.\]
In this variable, tetrahedra are counted \emph{with} multiplicity. That is, if an edge appears multiple times in the boundary of a given tetrahedron, this tetrahedron adds to its ``length'' each time. Note that
\[ \sum_{k\geq 0} k\cdot E_k = 6n.\]
We will also write
\[ E = \sum_{k\geq 0} E_k: \Omega_n \to \NN\]
for the total number of edges. 
The goal of this section is to study the distribution of $(E_k)_k$ as $n \to \infty$.

We will also count edges that we will call \emph{simple}. These are edges that neighbor each tetrahedron at most once. Figure \ref{pic_edge} shows an example. We will denote the number of simple edges by $E^\circ$ and the number of simple edges adjacent to $k$ tetrahedra by $E^\circ_k$. Note that $E_k^\circ = 0$ for all $k>n$, which is not at all necessary for non-simple edges.

\begin{figure}
\begin{center}
\begin{overpic}{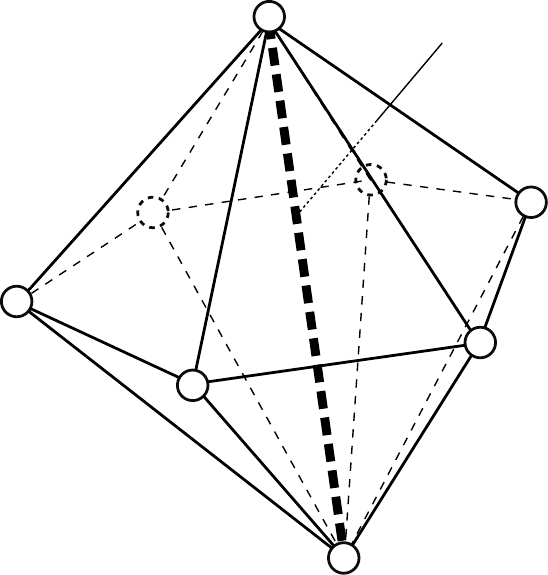}
\put(77.5,93) {$e$}
\end{overpic}
\caption{A simple edge $e$}\label{pic_edge}
\end{center}
\end{figure}

Concretely, we will prove the following estimates\footnote{When we write $k=o(f(n))$ for some function $f:\NN\to\NN$, what we mean is that the statement holds for any function $k:\NN\to\NN$ so that $k(n)=o(f(n)$.}
\begin{thm}[Combinatorics of edges]\label{thm_edges}
\begin{enumerate}[label={(\alph*)}]
\item \label{numtotal_edges} We have
\[ \EE_n[E] = \frac{1}{2} \log(n) + \bigO{1}\]
as $n\to\infty$.

\item \label{num_simple_edges_fixed_length} For all $k = o(\sqrt{n})$ we have
\[ \EE[E_k^\circ] = \frac{1}{2k} \cdot (1 + o(1)) \]
as $n\to \infty$. Moreover, the error is uniform in $k$.

\item \label{num_edges_fixed_length}For all $k=o(n)$ and $k_1\leq k_2\leq k$ we have
\[ \EE[E_k] \leq \frac{1}{2k}+\bigO{n^{-1}} \quad \text{and} \quad \EE\left[\sum_{l=k_1}^{k_2} E_l\right] \leq \frac{1}{2}\log(k_1/k_2)+\bigO{1} \]
as $n\to\infty$. In particular, for all $k=o(\sqrt{n})$ and $k_1\leq k_2\leq k$  we have
\[ \EE[E_k] = \frac{1}{2k} \cdot (1 + o(1)), \quad \EE\left[\sum_{l=k_1}^{k_2} E_l\right] = \frac{1}{2}\log(k_1/k_2)(1+o(1))\]
and 
\[ \EE[E_k-E_k^\circ] = o(1)\]
as $n\to\infty$.
\item\label{adjacent_edges}  For all $K,L=o\left(n^{1/3}\right)$ we have
\[ \EE[E_{KL}] = o(1) \]
as $n\to\infty$. Here $E_{KL}$ counts the number of pairs of edges of size $\leq K$ and $\leq L$ respectively that are incident to a common tetrahedron.
\end{enumerate}
\end{thm}

Before we get to the proof of this theorem, let us briefly note how to derive the Euler charateristic of $\partial N_n$ from it.

\begin{proof}[Proof of Theorem \ref{thm_topology}(b)]
Writing $v$, $e$ and $f$ for the number of vertices, edges and faces of the triangulation on $\partial N_n$, we have
\[ v= 2E \quad e=6n \quad \text{and} \quad f=4n.\]
As such, Theorem \ref{thm_edges}(a) implies our claim.
\end{proof}

\subsubsection{Peeling}
In order to prove Theorem \ref{thm_edges}(a), (c) and (d) we will use peeling (see for instance \cite{Curien}). Before we get to the proof, we need some preparation.

The main idea behind peeling is to build our random cell complex $\calT_n$ in a specific order. In particular, we will describe a peeling algorithm that determines a sequence of cell complexes
\[ \calT_n^{(0)},\; \calT_n^{(1)}, \ldots, \calT^{(2n)}_n \]
where $\calT_n^{(0)}$ consists of $n$ disjoint tetrahedra, $\calT^{(2n)}_n = \calT_n$ (in the sense that it has the same distribution) and in general $\calT^{(i+1)}_n$ can be obtained from $\calT^{(i)}_n$ by identifying exactly one pair of faces of $\calT^{(i)}_n$.

We will use two different peeling algorithms. The first to prove part (a) and the second to prove parts (c) and (d).

\subsubsection{Algorithm 1}
The first algorithm is very simple and closely resembles that of \cite[Section 8]{BM}.

\medskip
\fbox{ \begin{minipage}{14cm}
\textbf{Peeling algorithm 1}: \\
\underline{Initialisation}: \\
Objects: faces $f^{(0)}, f^{(1)},\ldots,f^{(2n)}$, $f'^{(0)},f'^{(1)},\ldots,f'^{(2n)}$ and $t\in\ZZ$.

\begin{itemize}
\item[-] Set $t = 0$. 
\item[-] Set $\calT_n^{(0)}$ equal to a disjoint union of $n$ tetrahedra. 
\item[-] Set $f^{(0)}$ equal to a face in $\partial \calT_n^{(0)}$, picked uniformly at random.
\end{itemize}

\underline{Iteration}: while $t< 2n$, repeat the following steps:
\begin{enumerate}
\item Glue the face $f^{(t)}$ to a uniformly random face $f'^{(t)}$ in $\partial \calT_n^{(t)} \setminus f^{(t)}$, with a uniformly random gluing. Call the result $\calT^{(t+1)}_n$.
\item If $t<2n$: Pick a uniformly random face $f^{(t+1)} \subset \partial \calT_n^{(t)}$      
\item Add $1$ to $t$.
\end{enumerate}
\end{minipage}}\medskip

Note that the distribution of $\calT_n^{(2n)}$ is the same as that of $\calT_n$.

\subsubsection{Closing off edges}

The reason for setting up the peeling algorithm is that we can now control the number of edges in $\calT_n$ by bounding the number of edges that are \emph{closed off} -- i.e. that disappear from the boundary -- during each step of the process. 

As such, let us define random variables $E^{(t)}$ that count the number of edges that are closed off when $\calT^{(t)}_n$ is created. This is the number of edges that lie in $\partial\calT^{(t-1)}_n$ but not in $\partial\calT^{(t)}_n$. Note that
\[ E = \sum_{t=1}^{2n} E^{(t)}.\]

Moreover, since any edge that gets closed off at the $t^{th}$ step necessarily lies in $f^{(t)}$, we have $E^{(t)} \leq 3$. One of the things we will argue below is that most of the time, we actually have $E^{(t)} \leq 1$. To this end, consider Figure \ref{pic_face}. It shows a schematic overview of the situation around the face $f^{(t)}$ at time $t$. Note that some of the faces $f_1$, $f_2$, $f_3$ and $f^{(t)}$ may coincide. However, if they don't, only one edge can be created at step $t$: $e_i$ is then closed off if and only if $f^{(t)}$ is glued to $f_i$ with exactly one out of the three possible face identifications.

\begin{figure}
\begin{center}
\begin{overpic}{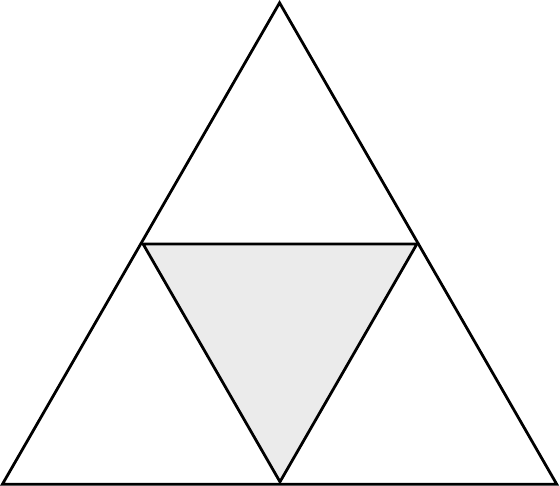}
\put(46,25) {$f^{(t)}$}
\put(22,15) {$f_1$}
\put(46,56) {$f_2$}
\put(70,15) {$f_3$}
\put(31,21) {$e_1$}
\put(46,46) {$e_2$}
\put(64,21) {$e_3$}
\end{overpic}
\caption{The face $f^{(t)}$ and its neighbors}\label{pic_face}
\end{center}
\end{figure}
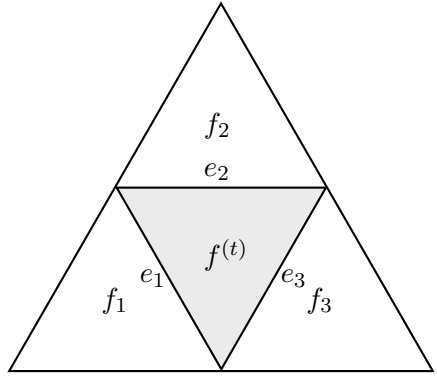

So, in what follows, we will make a distinction between \emph{singular faces} -- faces that are their own neighbor or of which some of the neighbors coincide -- and \emph{regular faces} -- faces that are not singular. We will write $F^{(t)}_{\mathrm{sing}}$ for the random variable that counts the number of singular faces in $\partial\calT_n^{(t)}$. Likewise, we define two sequences of random variables $E^{(t)}_{\mathrm{sing}}$,  $E^{(t)}_{\mathrm{reg}}$ where
\[
E^{(t)}_{\mathrm{sing}} = \left\{\begin{array}{ll} 
E^{(t)} & \text{if } f^{(t)} \text{ is singular} \\
0 & \text{otherwise.}
\end{array} \right. 
\quad \text{and} \quad
E^{(t)}_{\mathrm{reg}} = E^{(t)} - E^{(t)}_{\mathrm{sing}}.
\]
Since there are $4n-2t$ faces left when the $t^{th}$ step starts, we have
\begin{equation}\label{eq_regedges}
\EE[E^{(t)}_{\mathrm{reg}}\;|\; F^{(t)}_{\mathrm{sing}}] = \frac{4n-2t-F^{(t)}_{\mathrm{sing}}}{4n-2t} \frac{3}{3(4n-2t-1)} =  \frac{4n-2t-F^{(t)}_{\mathrm{sing}}}{(4n-2t)(4n-2t-1)} 
\end{equation}
and since at most three edges (those on $f^{(t)}$) are closed off during $t^{th}$ step and moreover there are at most $3$ choices for $f'^{(t)}$ and $3$ gluings per choice that result in an edge closure, we have:
\begin{equation}\label{eq_singedges}
 \EE[E^{(t)}_{\mathrm{sing}}\;|\; F^{(t)}_{\mathrm{sing}}]  \leq \frac{F^{(t)}_{\mathrm{sing}}}{4n-2t} \frac{3\cdot 3 \cdot 3}{3(4n-2t-1)} = \frac{9 \cdot F^{(t)}_{\mathrm{sing}}}{(4n-2t)(4n-2t-1)} . 
\end{equation}

So, we need to control $F^{(t)}_{\mathrm{sing}}$. 
\begin{lem}\label{lem_singfaces} For all $t\in\NN$ so that $t < 2n$.
\[ \EE_n\left[F^{(t)}_{\mathrm{sing}} \right] \leq 12 \cdot \log\left( \frac{2n}{2n-t} \right) + o(1) \]
as $n\to\infty$. The implied error is independent of $t$
\end{lem}

\begin{proof}  Let us write $\Delta F^{(t)}_{\mathrm{sing}}$ for the random variable that counts the number of singular faces that is created at step $t$ -- since we are only interested in an upper bound, we will ignore singular faces that disappear. Because $F^{(0)}_{\mathrm{sing}} = 0$, we have
\[ F^{(t)}_{\mathrm{sing}} \leq \sum_{s=1}^t \Delta F^{(s)}_{\mathrm{sing}} .\]
So, let us try to control $\EE[\Delta F^{(s)}]$. The only faces whose neighborhood changes during step $s$ are the neighbors of $f^{(s)}$ and $f'^{(s)}$. If such a neighbor of say $f^{(s)}$ is regular, it becomes singular only if one of its neighbors is also a neighbor of $f'^{(s)}$. Likewise, a regular neighbor of $f'^{(s)}$ becomes singular only if it shares a neighbor with $f^{(s)}$. In other words, $\Delta F^{(s)}$ can only be positive if the combinatorial distance -- the number of edges that needs to be traversed in $\partial\calT_n^{(s)}$ in order to move from one face to the other -- between $f^{(s)}$ and $f'^{(s)}$ is at most $3$. Note that this is independent of whether or not $f^{(s)}$ and $f'^{(s)}$ are regular.

Since there are at most $3\cdot 2\cdot 2 = 12$ faces at combinatorial distance at most $3$ from $f^{(s)}$, we have
\[ \EE[ \Delta F^{(s)}_{\mathrm{sing}}|\; F^{(s-1)}_{\mathrm{sing}} ] = \EE[\ind_{f^{(s)} } \Delta F^{(s)}_{\mathrm{sing}}|\; F^{(s-1)}_{\mathrm{sing}} ]
\frac{12}{4n-2s-1}.
  \] 
This implies that
\[\EE_n\left[ F^{(t)}_{\mathrm{sing}} \right] \leq  \sum_{1\leq t\leq 2n-1} \frac{12}{4n-2t-1} \leq 12 \sum_{k=4n-2t-1}^{4n-1} \frac{1}{k} = 12 \log\left(\frac{4n}{4n-2t} \right) + o(1) \]
as $n\to\infty$.
\end{proof}

\subsubsection{The total edge count}

We now have all the set up we need for the first part of Theorem \ref{thm_edges}:

\begin{proof}[Proof of Theorem \ref{thm_edges}(a)]
Fix any $\alpha \in (0,1)$. We have 
\[ \EE_n[E^{(t)}] = \EE_n\left[E^{(t)} \cdot \ind_{F^{(t)}_{\mathrm{sing}} \geq t^\alpha}\right] + \EE_n\left[E^{(t)} \cdot \ind_{F^{(t)}_{\mathrm{sing}} < t^\alpha} \right]\]
We start with the first term. By Lemma \ref{lem_singfaces}, \eqref{eq_regedges} and Markov's inequality we have
\[ 0 \leq \EE_n\left[E^{(t)}_{\mathrm{reg}} \cdot \ind_{F^{(t)}_{\mathrm{sing}} \geq t^\alpha}\right] \leq \frac{12 \log(t)}{n^\alpha} \frac{1}{4n-2t-1}, \]
for all $n$ large enough. Likewise, using Lemma \ref{lem_singfaces}, \eqref{eq_singedges} and Markov's inequality we obtain
\begin{equation}\label{eq_lotsofsingfaces}
 0 \leq \EE_n\left[E^{(t)}_{\mathrm{reg}} \cdot \ind_{F^{(t)}_{\mathrm{sing}} \geq t^\alpha}\right] \leq \frac{12 \log(t)}{n^\alpha} \frac{9}{4n-2t-1}.  
 \end{equation}
So we obtain
\[ \sum_{t=1}^{2n} \EE_n\left[E^{(t)} \cdot \ind_{F^{(t)}_{\mathrm{sing}} \geq t^\alpha} \right]  = \bigO{\frac{\log(n)^2}{n^\alpha}} \]
as $n\to \infty$. In words: most edges are created when few singular faces are present.

So, let us control this term. Again using \eqref{eq_regedges}, we have
\[ \frac{4n-2t-t^\alpha}{4n-2t} \frac{1}{4n-2t-1} \leq \EE_n\left[E^{(t)}_{\mathrm{reg}} \cdot \ind_{F^{(t)}_{\mathrm{sing}} < t^\alpha}\right] \leq  \frac{1}{4n-2t-1}.\]
Using \eqref{eq_singedges}, 
\[ 0 \leq \EE_n\left[E^{(t)}_{\mathrm{sing}} \cdot \ind_{F^{(t)}_{\mathrm{sing}} < t^\alpha}\right] \leq \frac{9 t^\alpha}{(4n-2t )(4n-2t-1)}. \]
These last two bounds give
\[ \sum_{t=1}^{2n} \EE_n\left[E^{(t)} \cdot \ind_{F^{(t)}_{\mathrm{sing}} < t^\alpha} \right]  = \sum_{t=1}^{2n} \frac{1}{4n-2t-1} + \bigO{1} = \frac{1}{2}\log(n) + \bigO{1} \]
Together with \eqref{eq_lotsofsingfaces}, this proves our claim.
\end{proof}

\subsubsection{Simple edges}

The proof of part (b) of our theorem -- the count of the expected number of simple edges -- will not use a peeling algorithm.

\begin{proof}[Proof of Theorem \ref{thm_edges}(b)]

We write
\[ \EE_n[E_k^\circ] = \sum_{(c_1\; c_2 \; \ldots \; c_k) } \EE_n[\ind_{(c_1\; c_2 \; \ldots \; c_k)}]  \]
where the sum runs over all cycles $(c_1\; c_2 \; \ldots \; c_k)$ 
of length $k$ so that
\begin{itemize}
\item[(i)] $c_i$ is a corner in some tetrahedron, i.e. a pair faces
\item[(ii)] and at most one corner of any given tetrahedron appears in the sequence.
\end{itemize}
Finally, given $\omega\in \Omega_n$,
\[ \ind_{(c_1\; c_2 \; \ldots \; c_k)} (\omega) = \left\{ \begin{array}{ll}
1 & (c_1\;c_2\; \ldots \; c_k) \text{ appears around an edge in }\omega \\
0 & \text{otherwise}
\end{array} \right. \]

It follows from (ii) that every cycle $(c_1\; c_2 \; \ldots \; c_k)$ in the sum corresponds to an identification of $k$ pairs of faces. As such
\[ \EE_n[\ind_{(c_1\; c_2 \; \ldots \; c_k)}] = \frac{1}{3^k(4n-1)(4n-3) \cdots (4n-2k+1)}.\]

So, all that remains is counting the number of possible cycles $(c_1\;c_2\;\ldots\;c_k)$ of corners. Writing $l$ for the number of tetrahedra of which we use two corners, we have
\[\card{\{(c_1\;c_2\;\ldots\;c_k)\}} = \frac{1}{2k} \binom{n}{k} \cdot 6^k \cdot 2^k \cdot k!.\]

The reason for this expression is as follows. First we count sequences instead of cycles:
\begin{itemize}
\item This gives a total of $\binom{n}{k}$ choices for the tetrathedra.
\item Per tetrahedron out of which a corner is used, we have a choice of $6$ corners. This gives rise to a factor $6^k$
\item Per corner that is used, we have two choices for the order in which the faces of that corner appear. This leads to a factor $2^k$
\item Finally, there are $k!$ ways to order the sequence.
\end{itemize}
Since we don't want to make a difference between sequences that differ by a cyclic permutation or are each others inverse, we divide by $2k$.

So we get
\begin{multline*}
 \EE_n[E_k^\circ] = \frac{1}{2k} \prod_{i=1}^k \frac{4n-4i+4}{4n-2i+1} \\ 
 = \frac{1}{2k} \exp\left(\sum_{i=1}^k \log\left(1-\frac{2i-3}{4n-2i+1}\right)\right) \\
  =  \frac{1}{2k}   \exp\left(- \sum_{i=1}^k \frac{2i-3}{4n-2i+1} + \bigO{n^{-2}}\right) \\
  =  \frac{1}{2k}  \exp(o(1))
 \end{multline*}
Where we used the fact that $k=o(\sqrt{n})$ in the last line.
\end{proof}

\subsubsection{Algorithm 2}

The second algorithm is actually a collection of algorithms, tailored towards counting the number of edges incendent to two edges in a fixed tetrahedron. As such, it starts peeling our random triangulation around a fixed starting edge $e$ and then continuous to peel around another fixed edge $e'$ once $e$ is closed.

\medskip
\fbox{ \begin{minipage}{14cm}
\textbf{Peeling algorithm 2}: \\
\underline{Input}: \\
A labelled tetrahedron $\tau$ and two fixed labeled oriented edges $e, e'\subset \tau$.

\underline{Initialisation}: \\
Objects: oriented edges $e^{(0)}, e^{(1)},\ldots, e^{(2n)}$, faces $f^{(0)}, f^{(1)},\ldots,f^{(2n)}$, $f'^{(0)},f'^{(1)},\ldots,f'^{(2n)}$ and $t\in\ZZ$.

\begin{itemize}
\item[-] Set $t = 0$. 
\item[-] Set $\calT_n^{(0)}$ equal to a disjoint union of $n$ tetrahedra, containing $\tau$. 
\item[-] Set $e^{(0)}=e$.
\end{itemize}

\underline{Iteration}: while $t< 2n$, repeat the following steps:
\begin{enumerate}
\item Glue the face $f^{(t)}$ to the right of $e^{(t)}$ to a uniformly random face $f'^{(t)}$ in $\partial \calT_n^{(t)} \setminus f$, with a uniformly random gluing. Call the result $\calT^{(t+1)}_n$.
\item \begin{itemize}
      \item[-] If $t+1<2n$ and $e^{(t)} \nsubseteq \partial\calT^{(t+1)}_n$:  
      \begin{itemize}
       \item If $e'\subset \partial\calT^{(t+1)}$, set $e^{(t+1)} = e'$. 
       \item Else: pick a uniformly random edge $e^{(t+1)} \subset \partial \calT_n^{(0)}$ and orient it randomly.
      \end{itemize}           
      \item[-] Else:  $e^{(t+1)} = e^{(t)}$
      \end{itemize}
\item Add $1$ to $t$.
\end{enumerate}
\end{minipage}}\medskip

Again note that the distribution of $\calT_n^{(2n)}$ is the same as that of $\calT_n$. Figure \ref{pic_initsetup} shows what the initial set up looks like.
\begin{figure}[!ht]
\begin{center}
\begin{overpic}{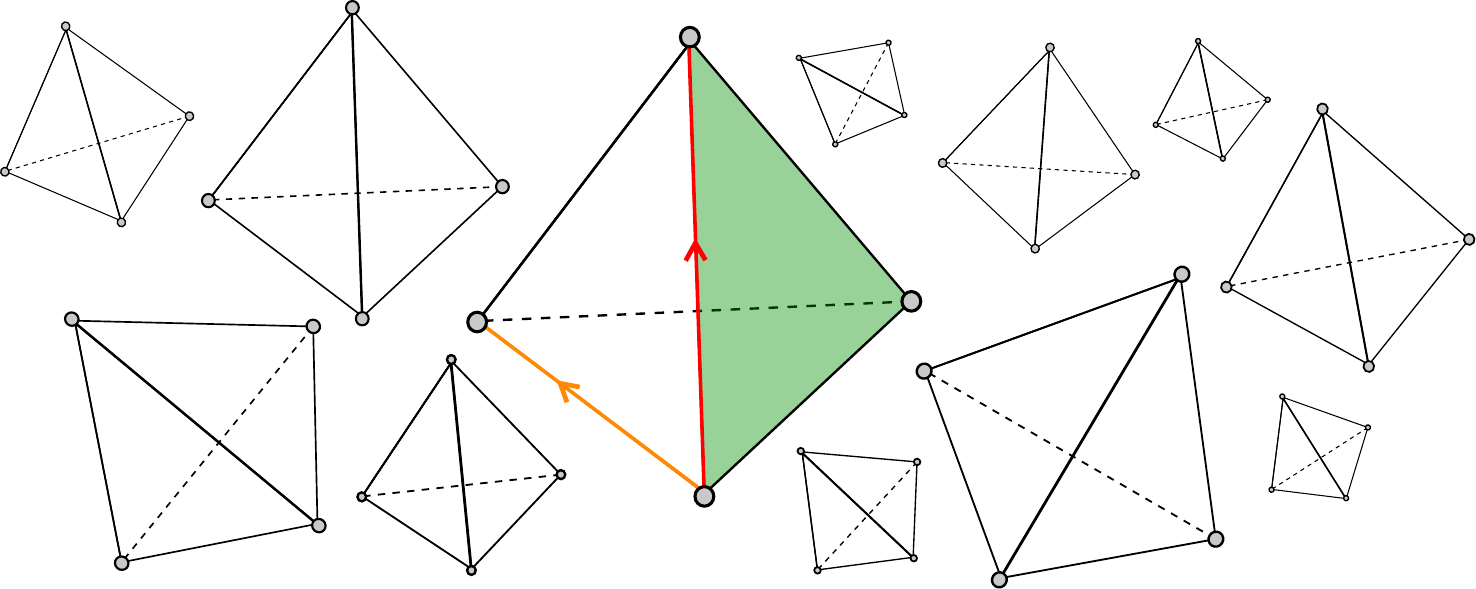}
\put(45,22) {$e$}
\put(38.5,14.5) {$e'$}
\put(50,21) {$f^{(0)}$}
\put(38,30) {$\tau$}
\put(90,2) {$\calT_n^{(0)}$}
\end{overpic}
\caption{The initial set up}\label{pic_initsetup}
\end{center}
\end{figure}
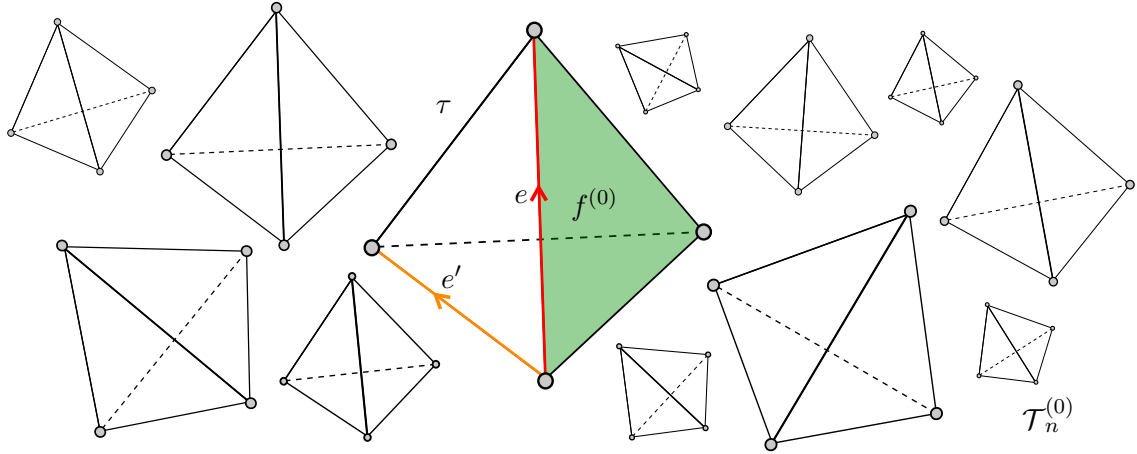

\subsubsection{All edges}

\begin{proof}[Proof of Theorem \ref{thm_edges}(c)]
Since we already have Theorem \ref{thm_edges}(b), we only need an upper bound on $\EE[E_k]$. In order to prove such a bound, we will write
\[ \EE[E_k] = \frac{1}{2k} \sum_e \EE[\ind_e^k]\]
where the sum runs over labelled oriented edges $e$ in our collection of $n$ tetrahedra (so the sum has $12n$ terms in total) and $\ind_e^k$ is the indicator for the event that $e$ is incident to exactly $k$ corners in the complex $\calT_n$. 

In order to bound  $\EE[\ind_e^k]$ from above, we use our second algorithm with $e$ as input. The second oriented edge $e'$ that the algorithm uses doesn't play a role in this proof, so we pick an arbitrary edge. We will also only care what happens in the first $k$ steps of the process.

Just like before, if during every step of the process, $f^{(t)}$ has three distinct neighbors, none of which is is $f^{(t)}$ itself -- i.e. if $f^{(t)}$ is regular --, it is easy to control the probability that our edge closes up in exactly $k$ steps. So, just like before, we need to bound the probability that in the first $k$ steps, our face becomes singular.

With the same argument as in Lemma \ref{lem_singfaces}, we have
\[\EE\left[\ind_{f^{(t)}\text{ is singular}}\right] \leq \frac{12}{4n-2t+1} + \EE\left[\ind_{f^{(t-1)}\text{ is singular}}\right].\]
In particular,
\[\EE\left[\ind_{f^{(k)}\text{ is singular}}\right] \leq \frac{12 k}{4n-2k+1}. \]

After $k-1$ gluings, there are $4n-2k+1$ faces left, we obtain, and even if a singular face is involved in the $k^{th}$ gluing, there are at most $3$ possible gluings that result in a closure. So we get:
\begin{multline*}
 \EE[\ind_e^k] = \EE\left[\ind_e^k\ind_{f^{(k)}\text{ is singular}}\right] + \EE\left[\ind_e^k\ind_{f^{(k)}\text{ is regular}}\right]  \\
 \leq \frac{1}{3(4n-2k+1)} + \frac{12 k}{4n-2k+1} \frac{3}{4n-2k+1}.
 \end{multline*}
This means that
\[
 \EE[E_k] \leq \frac{12n}{2k} \left( \frac{1}{3(4n-2k+1)} + \frac{12 k}{4n-2k+1} \frac{3}{4n-2k} \right), \]
 which proves our claim.
\end{proof}

\begin{proof}[Proof of Theorem \ref{thm_edges}(d)]
Let $k\leq K$ and $l\leq L$. Moreover, let $E_{kl}'$ denote the number of pairs of edges of sizes rexactly $k$ and $l$ respectively that are incident to a common tetrahedron. Just like in the proof above, we will write
\[ \EE[E_{kl}'] = \frac{1}{4kl} \sum_{e,e'} \EE[\ind_{e,e'}^{kl}]\]
where the sum runs over pairs of labelled oriented edges $e, e'$ that are incident to a single tetrahedron (so the sum has $12n\cdot 10$ terms in total) and $\ind_{e,e'}^{kl}$ is the indicator for the event that $e$ is incident to exactly $k$ and $e'$ to $l$ corners in the complex $\calT_n$. 

We again write
\begin{multline*}
\EE[\ind_{e,e'}^{kl}] =  \EE\left[\ind_{e,e'}^{kl}\ind_{f^{(k)}\text{ is singular}}\right] + \EE\left[\ind_{e,e'}^{kl}\ind_{f^{(k)}\text{ is regular and } f^{(k+l)} \text{ is singular}}\right] \\
 + \EE\left[\ind_{e,e'}^{kl}\ind_{f^{(k)} \text{ and } f^{(k+l)}\text{ are regular}}\right] .
 \end{multline*}
So we obtain, with exactly the same arguments as in the previous proof:

\begin{multline*}
\EE[\ind_{e,e'}^{kl}] \leq \frac{12k}{4n-2k+1}\frac{3}{4n-2k+1} \\ + \frac{1}{3(4n-2k+1)} \frac{12l}{4n-2k-2l+1} \frac{3}{4n-2k-2l+1} \\
 + \frac{1}{3(4n-2k+1)}\frac{1}{3(4n-2k-2l+1)} 
 \end{multline*}
So, multiplying this with $120 n$ and using that $k,l=o(n^{1/3})$ uniformly gives
\[ \EE[E_{kl}'] = o(n^{-2/3})\]
as $n\to \infty$, where the implied constant is uniform over $k,l$. Now summing over $k$ and $l$ gives
\[ \EE[E_{KL}] = o(1)\]
as $n\to\infty$.
\end{proof}


\subsection{Betti numbers} \label{betti_comb}

We give here the proof of our estimates on Betti numbers of $M_n$ (or $N_n$) in Theorem \ref{thm_main_topology}(d). First, a generating family for $H_1(M_n, \partial M_n)$ is given by images of the edges of $\mathcal T_n$. Applying Markov's inequality to \ref{thm_edges},(a) we get that
\[
b_1(M_n, \partial M_n) = o(\theta(n)) \quad \text{as }n\to\infty
\]
for any function $\theta:\NN\to\RR$ such that $\lim_{n\to\infty} \theta(n)/\log(n) = +\infty$, establishing the first estimate. From the connectedness of the boundary \ref{thm_topology},(a) it follows that with asymptotic probability 1 we get an exact sequence
\[
0 \to H_2(M_n) \to H_2(M_n, \partial M_n) \to H_1(\partial M_n) \to H_1(M_n) \to H_1(M_n, \partial M_n) \to 0.
\]
with asymptotic probability 1. Now this exact sequence and Poincar\'e duality (the ``half lives, half dies'' argument) imply that
\[
b_1(M_n) = \frac 1 2 b_1(\partial M_n) + b_1(M_n, \partial M_n)
\]
and together with the estimate for $b_1(M_n, \partial M_n)$, \ref{thm_topology},(b) we can conclude that with asymptotic probability 1 we have 
\[
\abs{b_1(M_n) - n} = o(\theta(n))
\]
for any function $\theta:\NN\to\RR$ that grows super-logarithmically, which is the second estimate. 


\subsection{Heegaard genus} \label{double_genus}

Here we prove the estimates on Heegaard genus of the double $\double M_n$ (or $\double N_n$) of Theorem \ref{thm_main_topology}(c), following an argument of Nathan Dunfield. Recall that $E$ is the number of edges in the triangulation $\mathcal T_n$ (equivalently the number of interior edges in the cellulation of $M_n$). We will first prove that
\begin{equation} \label{Nathan_trick}
  g(\double M_n) \le n + 1 + E
\end{equation}
which in view of Theorem \ref{thm_edges} implies the upper bound we are after.

To prove \eqref{Nathan_trick} we observe that $M_n$ minus regular neighbourhoods of its interior edges is a handlebody of genus $n+1$, as it is a regular neighbourhood of the dual graph to the cellulation of $M_n$ in truncated tetrahedra, which is a 4-valent graph on $n$ vertices. For each edge $e$ we write its regular neighbourhood $U_e$ as $D_e \times e$ where $D_e$ is a disc. We split it as $D_e = D_e^1 \cup D_e^2$ where $D_e^i$ are half-discs and we put $U_e^i = D_e^i \times e$; note that $M_n \cup \bigcup_e U_e^i$ is still a handlebody of genus $n+1$. We consider the two copies $M_n, \overline M_n$ of $M_n$ in $\double M_n$ (so $\double M_n = M_n \cup \overline M_n$), for a subset $W \subset M_n$ we denote by $\overline W$ its image in $\overline M_n$, and we put
\[
H_n^1 = \left( M_n \setminus \left( \bigcup_e U_e^1 \right)\right) \cup \left( \bigcup_e \overline U_e^2 \right)
\]
which is just $M_n \cup \bigcup_e U_e^1$ with 1-handles attached (one for each edge), so it is a handlebody of genus $n + E + 1$. Similarly
\[
H_n^2 = \left( \overline M_n \setminus \left( \bigcup_e U_e^2 \right) \right) \cup \left( \bigcup_e \overline U_e^1 \right)
\]
is a handlebody of the same genus. Now $\double M_n = H_n^1 \cup H_n^2$ and this proves \eqref{Nathan_trick}. 


For the lower bound we observe that the long exact sequence associated with $\partial M_n \to M_n \cup M_n \to \double M_n$ reduces to 
\[
0 \to H_2 (M_n) \oplus H_2(M_n) \to H_2(\double M_n) \to H_1(\partial M_n) \xrightarrow{i} H_1 (M_n) \oplus H_1(M_n) \to H_1(\double M_n) \to 0. 
\]
Since $b_2(\double M_n) = b_1(\double M_n)$ (Poincar\'e duality) and by our results on Betti numbers and genus of the boundary $2b_1(M_n) = b_1(\partial M_n)$ up to super-logarithmic error, it follows that
\[
b_1(\double M_n) = \mathrm{rank}(i) + \theta_0(n)
\]
with $\theta_0(n)$ super-logarithmic. As $i$ is diagonal embedding we have $\mathrm{rank}(i) \le b_1(M_n) = n$ up to super-logarithmic error, so we can conclude that $b_1(\double M_n) \ge n - \theta_1(n)$ with $\theta_1$ super-logarithmic. On the other hand $g(\double M_n) \ge b_1(\double M_n)$ and this proves the lower bound.

\section{Geometry}\label{sec_geom}

In this section we combine the combinatorial results from the previous section with hyperbolic geometry. Recall that $M_n$ denotes the compact manifold with boundary associated to $\calT_n$. Our main goal is to prove:

\begin{thm}[Geometry]
We have
\[ \lim_{n\to +\infty} \PP[M_n \text{ carries a hyperbolic metric with totally geodesic boundary}] = 1.\]
This metric has the following properties:
\begin{itemize}
\item[(a)]  The hyperbolic volume $\vol(M_n)$ of $M_n$ satisfies:
\[ \vol(M_n) \sim n\cdot v_O \quad \text{as } n\to\infty\]
in probability.
\item[(b)] There exists a constant $c_\lambda>0$ so that the first discrete Laplacian eigenvalue $\lambda_1(M_n)$ of $M_n$ satisfies
\[ \lim_{n\to +\infty} \PP[\lambda_1(M_n) > c_\lambda] = 1. \]
\item[(c)] There exists a constant $c_d>0$ such that the diameter $\diam(M_n)$ of $M_n$ satisfies:
\[
\lim_{n\to +\infty}\PP[\diam(M_n) < c_d \log(\vol(M_n))]  = 1
\]
\item[(d)] There exists a constant $c_s>0$ such that the systole  $\sys(M_n)$ of $M_n$ satisfies:
\[
\lim_{n\to +\infty}\PP[\sys(M_n) > c_s ]  = 1
\]
\item[(e)] For every $\eps>0$,
\[
\lim_{n\to +\infty}\PP\left[\sys(\double M_n) <  \frac{1}{n^{1/4-\eps}} \right]  = 1.
\]
The same holds for the minimal length among arcs in $M_n$ that are homotopically non-trivial relative to $\partial M_n$.
\end{itemize}
\end{thm}

We will prove this theorem in multiple steps. The first is hyperbolisation, which follows from Lemmas \ref{small_cusps} and \ref{large_cusps}). The asymptotic behaviour of the volume is then determined in Proposition \ref{volume_bounds} and the spectral gap is proven in Proposition \ref{compact_expansion}. We prove the bounds on the diameter and systole in Proposition \ref{prp_diam_sys}.

Finally, we will also prove that the Benjamini-Schramm limit of the sequence $(M_n)_n$ is the octatree.

\subsection{Random models for hyperbolic manifolds}

We will first describe two manifolds associated to an element $\omega\in\Omega_n$ in our probability space $\Omega_n$. The first is a cusped hyperbolic manifold $Y_n$ and the second is the manifold $M_n$ that we saw in the previous section, but now viewed as a Dehn filling of $Y_n$.

\begin{figure}
\begin{center}
\begin{overpic}{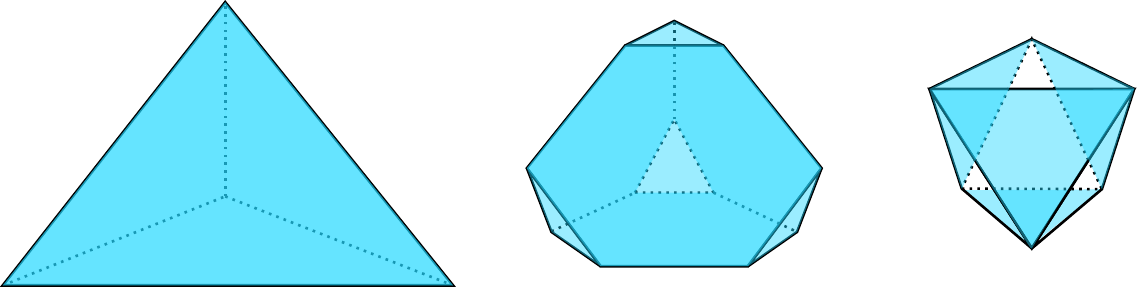}
\put (0,20){$\calT_n$} 
\put (43,20){$M_n$} 
\put (78,20){$Y_n$} 
\end{overpic}
\caption{The three building blocks for $\calT_n$, $M_n$ and $Y_n$ respectively. The faces that the gluing is performed along are shaded.}\label{pic_tetoct}
\end{center}
\end{figure}

Figure \ref{pic_tetoct} gives a topological picture of what is going on. We already associated a manifold $M_n$ to $\calT_n$ by truncating all the tetrahedra involved at their vertices. If we now contract the edges in the interior of this manifold and remove the resulting vertices, we obtain a new manifold $Y_n$ that is built out of a gluing of octahedra. The link of the octahedra's ideal vertices in this manifold are annuli, and we can fill them with cylinders to go back to the compact manifold. 

In what follows we describe this in some more detail.

\subsubsection{Manifolds with cusps and boundary}

Let $O$ be the ideal regular octahedron in $\HH^3$ (it can be realised as the convex hull of the vertices of a regular octahedron on the boundary at infinity $\mathbb S^2$). Its dihedral angles are right angles and its faces are ideal triangles. We orient each face with its outward normal. 

%
%

We take $n$ copies of $O$ which we label as follows: for each copy we attribute a label in $\{1,\ldots,4n\}$ to four of its faces so that no two of them are adjacent (and we ask that the labeling map be injective). Each of the unlabeled faces is then determined by the labels of the three faces adjacent to it and since it is orianted we can identify it with a 3-cycle on their labels. 

This setting is similar to that of \ref{delta-complex} and we can perform the same random construction from it: we partition the non-labeled faces uniformly randomly into pairs, and we glue the two faces in a pair in a uniformly randomly choses orientation-reversing way. 

The resulting octahedral complex is a non-compact manifold with boundary, and by endowing each $O$ with its hyperbolic structure we obtain that the result is a complete orientable hyperbolic manifold with totally geodesic boundary. We denote by $X_n$ the random hyperbolic manifold with boundary we constructed. We will also consider $Y_n$ where we condition on there not being any loops or bigons in the graph dual to the tesselation by octahedra. We record the hyperbolic structure in a lemma in order to be able to refer to it later on.

\begin{lem} \label{octa_gluing}
  The manifolds $X_n$ and $Y_n$ carry complete hyperbolic metrics of finite volume with totally geodesic boundary.
\end{lem}

We denote by $\Theta_n$ the (finite) set of all hyperbolic manifolds obtained by gluing $n$ octagons in this fashion. Let $Y$ be in some $\Theta_n$. Each cusp $c$ of $Y$ is tesselated by squares. Let $\ell(c)$ be their number (we will also call this the {\em length} of $c$) and for $k \in \NN$ let $B_k(Y)$ be the number of cusps of $Y$ with $\ell = k$.

\begin{lem} \label{cusp-edge}
  The random variable $C_k := B_k(Y_n)$ has the same distribution as the variable $E_k$ introduced at the beginning of \ref{delta-complex}.
\end{lem}


\subsubsection{Compact manifolds}

Recall from Section \ref{delta-complex} that $N_n$ is a manifold with boundary obtained from randomly gluing truncated tetrahedra along their faces. Moreover, $M_n$ is a random manifold that has the distribution of $N_n$, conditioned on the dual graph $\calG_n$ not having any loops and multiple edges.

Let us now describe how $M_n$ (and $N_n$ respectively) can be obtained from $Y_n$ (and $X_n$ respectively) via Dehn filling.

Let $Y \in \Theta_n$. Its boundary $S$ is a hyperbolic surface with cusps. Moreover there is a pairing on the cusps where we associate two cusps of $S$ if they are asymptotic in $Y$. If we remove a horospherical neighbourhood of each cusp we obtain a compact manifold $\ovl Y$ whose boundary is made up of $S$ together with closed annuli linking paired cusps, and by the thick-thin decomposition $Y$ is homeomorphic to $\ovl Y$ minus the annuli. We can then perform surgery on $\ovl Y$ as follows:\footnote{Another way to describe it is that it is the restriction to $Y$ of the unique Dehn surgery on the double $\mathcal DY$ which is equivariant with respect to the reflection in $\pl Y$. } to each annulus we glue a cylinder $D \times [0, 1]$ along the boundary $[0, 1] \times \pl D$. We obtain a compact manifold $M$ with boundary $\ovl S$. We denote by $\Xi_n$ the set of such manifolds $M$ obtained from $Y \in \Theta_n$. 

Note that $M_n$ (and $N_n$ respectively) has the same distribution as the Dehn filling of $Y_n$ (and $X_n$ respectively) described above. We again record this in a proposition to be able to refer to it later:

\begin{prop} \label{same_models}
  The variables $ M_n$ (respectively $N_n$) and the Dehn filling of $Y_n$ (respectively $X_n$) described above have the same distribution. 
\end{prop}

%
%
%


\subsection{Bounds for Dehn surgeries and hyperbolicity}

In this section we prove that $M_n$ is hyperbolic with asymptotic probability 1 (this is part of the statement of Lemma \ref{large_cusps}) and we give precise bounds for the variation in geometry between $Y$ and $M$.

Our proof goes in two steps. First we use Andreev's theorem to control what happens when ``small'' cusps of $Y$ are filled and after this we use recent results by Futer--Purcell--Schleimer to control the change in geometry when the ``large'' cusps  are filled.

\subsubsection{Andreev's Theorem} \label{andreev}

To construct explicit hyperbolisations we will need Andreev's theorem describing acute-angles polyhedra in $\HH^3$. We refer to \cite{Roeder_Hubbard} for a proof of this result. Before giving the statement we recall that given a combinatorial 2-polyhedron $P$, with dual graph\footnote{Recall that this is the graph on 2-dimensional faces of $P$ with an edge between two faces for each edge they share. } $H$, a circuit in $H$ is said to be {\em prismatic} if for any edge in the circuit, the endpoints of the corresponding edge of $P$ are distinct. The following is a combination of Theorem 1.4 and Proposition 1.5 in \cite{Roeder_Hubbard}:

\begin{thm}[Andreev's Theorem] \label{thm_andreev}
  Let $P$ be an abstract polyhedron with at least six faces, and $\alpha$ a function from edges of $P$ to $(0,\pi/2]$. Then there exists a realisation of $P$ in $\HH^3$ which is of finite volume and whose dihedral angles are given by $\alpha$ if and only if the following conditions are satisfied.
  \begin{itemize}
  \item For any three edges $e_1, e_2, e_3$ meeting in one vertex we have $\sum \alpha(e_i) \ge \pi$ (equility ocurring if and only if the vertex is ideal).

  \item If $(e_1, \ldots, e_k)$ is a prismatic $k$-circuit with $k = 3, 4$ then $\sum \alpha(e_i) < (k-2)\pi$.

  \end{itemize}
\end{thm}

\subsubsection{Filling small cusps}

We start by filling the small cusps of $Y_n$. These will be cusps of length $\leq n^{1/4}$. Note that the resuling manifold, that we call $Z_n$ is rigid -- i.e. if we can find a complete hyperbolic metric of finite volume on it, it's unique up to isometry -- by Mostow--Prasad rigidity of its double.

\begin{lem} \label{small_cusps}
  There exists $J_0 > 0$ such that the following holds for any Margulis constant $\delta > 0$ and any $\eps > 0$. For any $Y \in \Theta_n$,
  \begin{itemize}
  \item let $c_1, \ldots, c_m$ be the cusps of $Y$ of length at most $n^{1/4}$,
  \item let $Z$ be the manifold obtained by filling $c_1, \ldots, c_m$,
  \item let $Z_1$ be the union of all octahedra of $Y$ containing one of the $c_i$ and $Z_2$ its complement.
\end{itemize}
Then with probability at least $1-\eps$ in the model $Y_n$ for $n$ large enough, we have that: 
  \begin{itemize}
  \item The $\delta$-thick part of the image\footnote{We view $Y$ as the complement of the core arcs in $Z$.} of $Z_1$ in $Z$ is $J_0$-bilipshitz to that of $Z_1$; 
  \item The image of $Z_2$ in $Z_\sigma$ is isometric to $Z_2$. 
  \end{itemize}
\end{lem}

\begin{proof}
  This is the (only) part of our proof of hyperbolisation that will use the assumption that the dual graph $\calG_n$ is simple.

  Let $O_1, \ldots, O_n$ be the octahedra tesselating $Y$ and $O_i$,  $1 \le i \le k$ those containing a cusp $c_i$ with $\ell(c_i) \le n^{1/4}$ and $O_{k+1}, \ldots, O_n$ the remaining ones. By Theorem \ref{thm_edges}\ref{adjacent_edges} we may assume that each $O_i$ contains exactly one such cusp. We have $Z_1 = O_1 \cup \cdots \cup O_k$ and $Z_2 = O_{k+1} \cup \cdots \cup O_n$. Then the part of boundary of $Z_1$ and $Z_2$ along which they are glued is a disjoint union of ideal regular squares.

  To construct the hyperbolic structure on the filled manifold we replace $O_1, \ldots, O_k$ by polyhedra constructed as follows. First we assume that $\ell(c_i) \ge 4$. Consider the following polyhedron, which is an octahedron on which 1 vertex has been replaced by an edge (marked red in the picture): 
  \begin{center}
    \includegraphics[width=.3\textwidth]{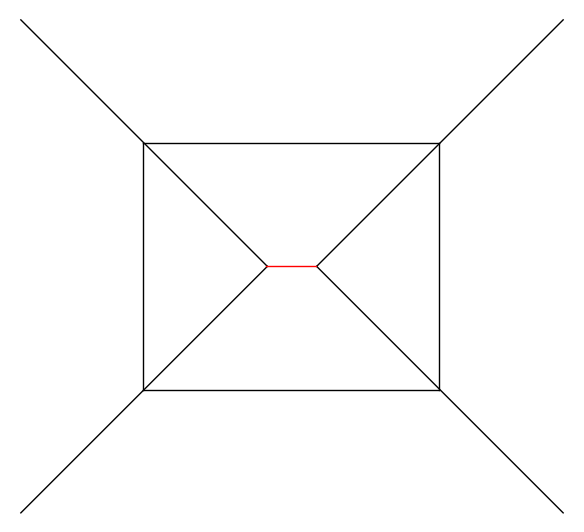}
  \end{center}
  There are no prismatic 3- or 4-circuits in the dual graph so it follows from Andreev's Theorem (Theorem \ref{thm_andreev}) that for $l \ge 4$, this has the structure of an hyperbolic polyhedron $P_l$ with right angles at all edges except the red one which has angle $2\pi/l$ (we need $l \ge 4$ for this not to be obtuse).

  If $\ell(c_i) = 3$ then we can still construct $P_3$ as follows: the combinatorial polyhedron has a symmetry along the red edge, which decomposes it as the double of the following polyhedron along the blue-colored face : 
  \begin{center}
    \includegraphics[width=.3\textwidth]{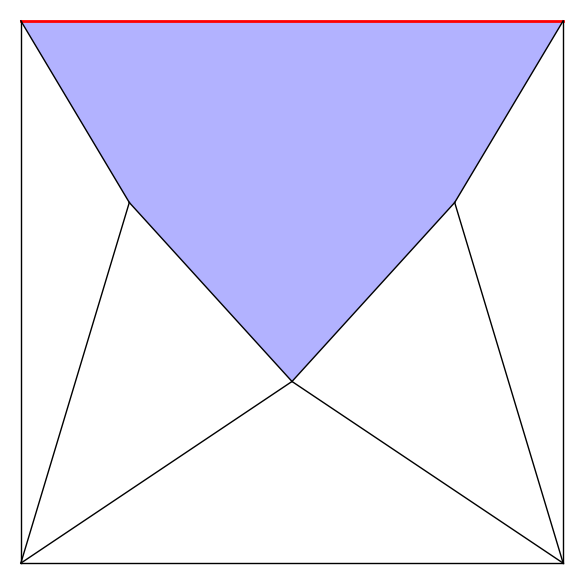}
  \end{center}
  and the latter has no prismatic circuits, so it admits a hyperbolic structure with right angles on the black edges and $\pi/3$ on the red edge by Andreev's theorem.

  \medskip

  For $l \ge 4$ let $Q_l$ be the hyperbolic manifold obtained by gluing $l$ copies of $O$ in a circular pattern, along disjoint faces sharing an ideal vertex. The faces opposite to the glued faces form a union of disjoint ideal triangles in $Q_l$. Dehn surgery on $Q_l$ amounts to replacing each copy of $O$ by a copy of the polyhedron $P_l$ (the edge with angle $2\pi/l$ replacing the ideal vertex on which surgery is done), to obtain a polyhedron $Q_l'$ whose boundary is two $l$-gons, $l$ regular ideal squares and $2l$ ideal triangles meeting at right angles. This is illustrated in the following figure, where these are colored blue and the central edge red : 
  \begin{center}
    \includegraphics[width=.5\textwidth]{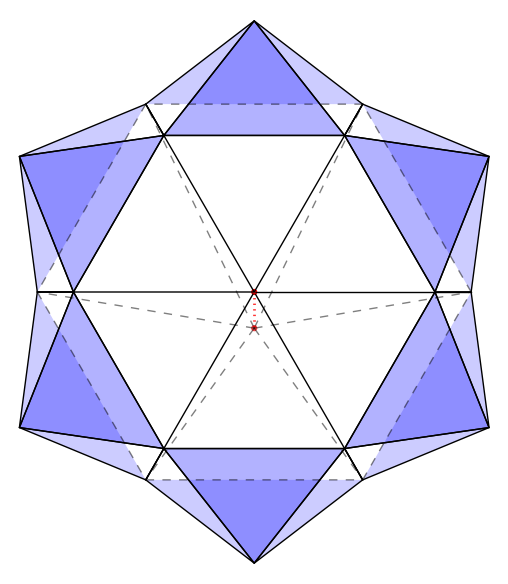}
  \end{center}
  Now if $M$ is generic in the sense of \ref{thm_edges}\ref{adjacent_edges} and furthermore all edges of length at most $K(n)$ (where $K(n)$ is any $o(n^{1/4})$) are simple\footnote{We could also have realised the surgery by explicit polygons for non-simple edges but we found this argument to be simpler.} (which is generic by \ref{thm_edges}\ref{num_simple_edges_fixed_length},\ref{num_edges_fixed_length}) then $Z_1$ is a disjoint union of $Q_l$s and filling the small cusps amounts to replacing each of these with a $Q_l'$. In particular we can glue the rest of the octahedra in the pattern given by $\mathcal G$ to obtain $Z$. This proves that the image of $Z_2$ in $Z$ is isometric to $Z_2$. Since $P_l$ converges in Gromov--Hausdorff topology (pointed anywhere in the thick part) to $O$ the $\delta$-thick parts of $P_l$ are uniformly (independently of $l$) bilipschitz to the $\delta$-thick part of $O$, and it follows immediately that $Z_1$ is uniformly (independently of generic $\mathcal G$) bilipschitz to its image in $Z$. 
\end{proof}


\subsubsection{Filling large cusps} 

In order to fill the cusps remaining in $Z$, we rely on results of Futer--Purcell--Schleimer. Again using $Z$ to denote the manifold obtained from $Y$ by filling its msall cusps, we will think of $Y\subset Z \subset M$. We have:

\begin{lem} \label{large_cusps}
  Let $Y, Z$ be as in Lemma \ref{small_cusps}. Then, for any $\delta > 0$ (smaller than the Margulis constant for $\HH^3$) and any $\eta > 0$, with probability for the $Y_n$ model converging to 1 as $n \to +\infty$ the following holds:
\begin{itemize}
 \item There exists Riemannian metrics $g_0, g_{2\pi}$ on $M$ such that $(Y, g_0)$ is the complete hyperbolic structure and the completion of $(M, g_{2\pi})$ is a compact hyperbolic manifold with totally geodesic boundary which is diffeomorphic to the Dehn filling of $Y$. 

\item Moreover $Z_{\ge 6\delta/5} \subset M_{\ge \delta} \subset Z_{\ge 5\delta/6}$ and these inclusions are $(1+\eta)$-lipschitz. 
\end{itemize}
\end{lem}

\begin{proof}
  Let $c_{m+1}, \ldots, c_h$ be all remaining cusps in $Z$ (recall that $c_1, \ldots, c_m$ were the cusps of length at most $n^{1/4}$). Realising them as arbitrary horosphere quotients in $Z$ these cusps have an area $a_j$ and a length of the vertical curve $l_j$. Let $L_j = l_j/\sqrt{a_j}$; this does not depend on the arbitrary choice of horospheres (as long as their quotients are homeomorphic to $2$-tori). Following \cite[Definition 1.3]{Futer_Purcell_Schleimer} we define $L>0$ by 
  \[
  \frac 1 {L^2} = \sum_{j=m+1}^h \frac 1 {L_j^2}. 
  \]
  For $m+1 \le j \le h$ let $k_j$ be the number of cusps $c_i$, $1 \le i \le m$, which share an octahedron with $c_j$ -- in other words, the number of small cusps that share an octahedron with $c_j$. We claim that with probability tending to 1 we have $\max(k_j) \ll \frac{\log(n)}{n^{1/24}} \ell(c_j)$. 
  
  To prove this we separate two cases: first, when $n^{1/4} \le \ell(c_j) \le n^{7/24}$ (note $7/24 = 1/3 - 1/24 = 1/4+1/24$) we have by \ref{thm_edges}\ref{adjacent_edges} that with probability tending to 1 we have $k_j = 0$ for all these $j$. In the remaining cases we have that $k_j \ll h \cdot n^{1/4}$ and since $h = E$ (the number of edges in the original triangulation) and $E \ll \log(n)$ by \ref{thm_edges}\ref{numtotal_edges} it follows that if $\ell(c_j) \ge n^{7/24}$
  \[
  k_j \ll \log(n)n^{1/4} \le \frac{\log(n)}{n^{1/24}} \ell(c_j)
  \]
  which finishes the proof of the claim.

  Now it follows from Lemma \ref{small_cusps} if $m+1 \le j \le h$ we have at least $\ell(c_j) - k_j$ regular squares in the tesselation of $c_j$ in $Z$, and so
  \[
  \max_{m+1 \le j \le h}\left(\frac{1}{L_j^2}\right) \le \max_{m+1 \le j \le h} \left(\frac 1{\ell(c_j) - k_j}\right) \ll \frac 1{n^{1/4}}. 
  \]
  Using \ref{thm_edges}\ref{numtotal_edges} again we get that with asymptotic probability 1 we have
  \[
  \sum_{j=m+1}^{h} \frac 1{L_j^2} \le \frac{\alpha \log(n)^2}{n^{1/4}}
  \]
  in particular $L^2 \ge n^{1/8}$ with asymptotic probability 1. With this our lemma is an immediate consequence of \cite[Theorem 9.28]{Futer_Purcell_Schleimer} as this implies that for any fixed $\delta$, with asymptotic probablity 1 we have that $Z$ satisfies the hypothesis (9.30) in this statement (with $\epsilon = \delta$). 
\end{proof}


\subsection{Expansion}

For a non-necessarily compact Riemannian manifold $V$ we denote by $\lambda_1(V)$ the bottom of the discrete spectrum of the Laplace--Beltrami operator of $V$ (if $V$ has non-empty boundary we take it to be the minimum between the spectra with Neumann or Dirichlet conditions). Using the results from the preceding section and comparison results in spectral geometry due to Mantuano and Hamenst\"adt we prove the following. 

\begin{prop} \label{compact_expansion}
  There exists $c_\lambda > 0$ such that: 
  \[
  \lim_{n \to +\infty} \mathbb P[\lambda_1(M_n) \ge c_\lambda] = 1. 
  \]
\end{prop}

\begin{proof}
One way to prove Proposition \ref{compact_expansion} would be a minor modification of the argument in \cite[Section 4]{BrooksApollonian}, based on the Cheeger constant of $M$ (see also the appendix to \cite{Breuillard_lecture}). We will instead work with the double $N = \mathcal DM$ to be able to apply directly the results by Hamenst\"adt and Mantuano.

In this proof we work with a $M \in \Xi_n$ which carries a hyperbolic metric with totally geodesic boundary (which we proved happens with asymptotic probability 1). 


Let $N = \mathcal DM$ be the double of $M$ along its (totally geodesic) boundary, which is a closed hyperbolic manifold. The space $L^2(\mathcal DM)$ decomposes into the direct sum of $\pm 1$ eigenspaces for the symmetry in $\partial M$ and these spaces correspond to spaces of functions on $M$ satisfying Neumann or Dirichlet conditions on $\pl M$. So we have that $\lambda_1(\mathcal DM) = \lambda_1(M)$ and in the rest of the proof we will be concerned with establishing that $\lambda_1(\mathcal D X)$ is bounded away from 0. 


We fix a Margulis constant $\delta$. By \cite[Theorem 1]{Hamenstaedt_comp} we have that $\lambda_1(N) > \lambda_1(N_{\ge \delta})/3$ (or $\lambda_1(N)$ is uniformly bounded away from zero, in which case we are finished). So we must bound $\lambda_1(N_{\ge \delta})$ from below. 

To do so we will use the following result which is an immediate application of \cite[Theorem 3.7]{Mantuano}: 
\begin{itemize}
\item let $V$ be a compact hyperbolic 3--manifold with $\inj(V) \ge \delta$ (for example the $\delta$-thick part of a manifold of finite volume if its boundary is smooth),
\item Let $\mathcal X$ be a maximal $\delta/2$-separated subset of $V$, on which we put the graph structure where there is an edge between $x, y \in \mathcal X$ if they are at distance at most $2\delta$ from each other in $V$ ($\mathcal X$ is called a discretisation of $V$, and $\delta$ its mesh). 
\end{itemize}
Then there is $c > 0$ depending only on $\delta$ such that $\lambda_1(V) \ge c\lambda_1(\mathcal X)$.

We record the following well-known facts which we will use to compare between discretisations of our manifolds $Y, Z$ and $M$

\begin{lem} \label{disc_qi}
  Let $E_1, E_2$ be two metric geodesic spaces and $\mathcal X_i$ a discretisation of $E_i$. 
  \begin{enumerate}
  \item The inclusion $\mathcal X_i \subset E_i$ is a quasi-isometry with constants depending only on the mesh.
    
  \item If $\varphi$ is a quasi-isometry from $E_1$ to $E_2$ and $q$ is a nearest-neighbour projection from $E_2$ to $\mathcal X_2$ then $q\circ\varphi$ induces a quasi-isometry from $\mathcal X_1$ to $\mathcal X_2$, whose quasi-isometry constants depend only on those of $\varphi$ and on the meshes of $\mathcal X_1, \mathcal X_2$.
  \end{enumerate}
\end{lem}

\begin{proof}
  To prove the ``quasi-isometric embedding'' part of the first statement take $x, x' \in \mathcal X_i$, then the nerve of a cover of a geodesic in $E_i$ between $x, x'$ by $\delta$-balls (where $\delta$ is the mesh) centered in $\mathcal X_i$ gives a path in $X_i$ with length at most $2\delta^{-1}d_{E_i}(x, x')$; the reverse inequailty is immediate. It is also immediate to check that a quasi-inverse is given by any nearest-point projection, which is a quasi-isometry whose constants also depend only on the mesh. The second point immediately follows. 
\end{proof}

Let $\mathcal G_1$ be a discretisation of $N_{\ge\delta}$ with mesh $\delta/2$. Let $\varphi : N_{\ge\delta} \to \mathcal DY_{\ge\delta}$ be defined as follows: by Lemma \ref{large_cusps} we have $M_{\ge\delta} \supset Z_{\ge 6\delta/5}$, so we can define a retraction $\pi_1 : M_{\ge\delta} \to Z_{\ge 6\delta/5}$ by following the geodesic flow in the direction orthogonal to the boundary $\pl Z_{\ge 5\delta/6}$---if there are multiple possible directions to follow, i.e. if we are on a core geodesic, we choose one arbitrarily. We extend this to $N_{\ge\delta} = \mathcal DM_{\ge \delta}$ by symmetry. By the rest of the statement of the lemma, this is $(1+\eta)$-bilipschitz on $N_{\ge\delta}$ and since $M_{\ge\delta} \subset Z_{\ge 5\delta/6}$, for $x \in N_{\ge\delta} \setminus \mathcal D Z_{\ge 6\delta/5}$ we have $d_M(x, \mathcal DZ_{\ge 6\delta/5}) \le a\delta$ for an absolute $a$. It follows that $\pi$ is a $(1+\eta, b\eta\delta)$-quasi-isometry, for some absolute $b$. Now we extend by symmetry the $J_0$-bilipschitz map $\psi$ from $Z_{\ge 6\delta/5}$ to $Y_{\ge 6\delta/5}$ given by Lemma \ref{small_cusps} and put $\varphi = \psi \circ \pi$, which from what we said is a quasi-isometry from $N_{\ge\delta}$ to $\mathcal D Y_{\ge 6\delta/5}$ with constants depending only on $\delta, \eta$. 

Applying the lemma to $\varphi$ we get that $\mathcal G_1$ and an arbitrary discretisation $\mathcal G_2$ of $Y_{\ge 6\delta/5}$ with mesh $\delta/2$ are quasi-isometric to each other, with constants depending only on $\delta$. Let $\mathcal{DG}$ be the graph dual to the tesselation of $N$ by octahedra; it is obtained by taking two copies of $\mathcal G$ and adding four edges between every pair of corresponding vertices. On $\mathcal G_2$ we can define a map to $\mathcal{DG}$ by mapping all vertices in a given octahedron of $M$ to the center of that octahedron (we choose arbitarily for vertices on the boundary between two faces). This is a quasi-isometry with constants depending only on $\delta$ (via the diameter of $O_{\ge \delta}$). Composing $\varphi$ with this we get a quasi-isometry from $\mathcal G_1$ to $\mathcal G$. By \cite[Theorem 2.1]{Mantuano} it follows that $\lambda_1(\mathcal G_1) \ge c'\lambda_1(\mathcal{DG})$ where $c$ depends only on $\delta$. As $\mathcal G_1$ is a discretisation of $\mathcal DX_{\ge\delta}$, by loc.~cit., Theorem 3.7 (see the statement at the beginning of the section) it finally follows that $\lambda_1(\mathcal DX_{\ge\delta}) \ge c''\lambda_1(\mathcal{DG})$. It is well known that $\mathcal G_n$ is an expander asymptotically almost surely; the sharpest bounds on its spectral gap are due to Friedman \cite{Friedman}. The double $\mathcal{DG}$ is quasi-isometric to $\mathcal G$ with uniform constants via the inclusion, so loc. cit., Theorem 2.1 gives us that it is also an expander a.a.s. We conclude that $\lambda_1(N_{\ge\delta})$, and hence also $\lambda_1(N)$, is bounded away from zero. 
\end{proof}


\subsection{Volumes}

If the manifold $X \in \Xi_n$ is hyperbolic then it has a hyperbolic volume $\vol(X)$. Otherwise we take $\vol(X) = 0$. Recall that $v_O$ denotes the volume of the right-angled hyperbolic octahedron. 

\begin{prop} \label{volume_bounds}
 We have 
\[ \vol(M_n) \sim n\cdot v_O \quad \text{as } n\to\infty\]
in probability.
\end{prop}

\begin{proof}
  If $M \in \Xi_n$ is hyperbolic then it is a Dehn surgery on a manifold $Y \in \Theta_n$. The latter is a union of $n$ copies of the octahedron $O$. As the hyperbolic volume decreases under hyperbolic Dehn surgery we get that $\vol(X) \le nv_O$. 

  All statements in the following paragraph hold asymptotically almost surely. By Lemma \ref{small_cusps} we have that
  \begin{equation} \label{compvol1}
    \vol(Z) \ge \vol(Y) - O(n^{1/4}\log n) = n\vol(O) - O(n^{1/4}\log n) 
  \end{equation}
  (at most $4n^{1/4}\log n$ octahedra are changed from $Z$ to $Y$ since this is an upper bound for the number of squares in the small cusps in a generic $Y$ by Theorem \ref{thm_edges}\ref{numtotal_edges}). By Lemma \ref{large_cusps} we have that for any positive $\delta$ and $\eta$ we have, since if two riemannian metrics on a manifold are $\eta'$-blilipschitz to each other ther the volume forms are $O(\eta)$ pointwise close to each other, that: 
  \[
  \vol(M) \ge (1 - c\eta)\vol(Z_{\ge 6\delta/5})
  \]
  for some $c>0$ independent of $\delta, \eta$. The thin part of $Z$ is made of $O(\log(n))$ tubes coming from the Dehn filling of small cusps, so constibuting a volume $O(\log n)$, and the rest is cusps. For a cusp $C$ we have $\vol(C) = \mathrm{Area}(\pl C)$, and the boundary of the cusps of $Z_{\ge 6\delta/5}$ is made of $n - O(\log n)$ euclidean squares with edge length $O(\delta)$. It follows that
  \begin{equation}\label{compvol2}
  \vol(M) \ge (1 - c\eta)\vol(Z_{\ge 6\delta/5}) \ge (1 - c\eta)\vol(Z) - \delta O(n). 
  \end{equation}
  Taking $\delta$ and $\eta$ to 0 we get the statement we want from \eqref{compvol1} and \eqref{compvol2}. 
\end{proof}


\subsection{Diameter and systole}

Lemmas \ref{small_cusps} and \ref{large_cusps} above together with our combinatorial bounds (Theorem \ref{thm_edges}) and results by Futer--Purcell--Schleimer and Bollob\'as--Fernandez-de-la-Vega imply the following bounds on the diameter and systole of $M_n$ and $\double M_n$:

\begin{prop}\label{prp_diam_sys}
\begin{itemize}
\item[(a)] There exists a constant $c_d>0$ such that the diameter $\diam(M_n)$ of $M_n$ satisfies:
\[
\lim_{n\to +\infty}\PP[\diam(M_n) < c_d \log(\vol(M_n))]  = 1
\]
\item[(b)] There exists a constant $c_s>0$ such that the systole  $\sys(M_n)$ of $M_n$ satisfies:
\[
\lim_{n\to +\infty}\PP[\sys(M_n) > c_s ]  = 1
\]
\item[(c)] For every $\eps>0$,
\[
\lim_{n\to +\infty}\PP\left[\sys(\double M_n) <  \frac{1}{n^{1/4-\eps}} \right]  = 1.
\]
The same holds for the minimal length among arcs in $M_n$ that are homotopically non-trivial relative to $\partial M_n$.
\end{itemize}
\end{prop}

\begin{proof}
We start with item (a). Bollob\'as--Fernandez-de-la-Vega \cite{BFdlV} proved that the diameter (in the graph distance) of a random $4$-regular graph $\calG_n$ on $n$ vertices satisfies
\[
\diam(\calG_n) \leq \log_3(n) + o(\log(n))
\]
in probability. 

Again, using results from graph theory \cite{Bol1,Wor}, we may assume that $\calG_n$ is conditioned to not have loops or multiple edges, so that $\calG_n$ is uniformly quasi-isometric to the $\delta$-thick part of $Y_n$ (with constants that only depend on $\delta$). 

Now we pick $\delta>0$ smaller than the Margulis constant for $\HH^3$. Using Lemmas \ref{small_cusps} and \ref{large_cusps}, plus the fact that the polytopes $P_l$ descibed in the proof of Lemma \ref{small_cusps} converge to $O$, this implies that the $\delta$-thick part $(M_n)_{\geq \delta}$ of $M_n$ is uniformly quasi-isometric to $\calG_n$. Hence, there exists a constant $C_\delta>0$ such that
\[
\diam((M_n)_{\geq \delta}) \leq C_\delta \log(n)
\]
asymptotically almost surely. 

In order to control the diameter of the thin parts of $M_n$, it's easier to think in terms of Margulis tubes, so we will consider the double $\double M_n$ as a Dehn filling of $\double Y_n$. The Margulis Lemma tells us that the thin part of $\double M_n$ consists of standard tubes (see for instance \cite[Chapter D]{BP}) of the form
\[
T_r = \st{x\in \double M_n}{ \dist(x,\gamma) < r}
\]
where $\gamma$ is a simple closed geodesic. As such, the diameter of such a tube is at most $2r+\ell(\gamma) \leq 2r + \delta$. 

The length of a meridian on the boundary torus of a standard tube is $2\pi \sinh(r)$. In $Y_n$, the length of the meridian is a constant multiple of the combinatorial length of the corresponding cusp (and hence bounded by $6n$). 

We want to estimate the lengths of meridians in $M_n$ in terms of those in $Y_n$. To do so we first observe that these lengths are the same between $Y_n$ and $Z_n$ by Lemma \ref{small_cusps}. Then using Lemma \ref{large_cusps} in the same way that we used to construct a retraction in the proof of Proposition \ref{compact_expansion} we see that there is a bilipschitz map (with constants independent of $n$) between the boundaries of $M_{\ge\delta}$ and $Z_{\ge\delta}$, which sends meridian to meridian. It follows that there exists a constant $D_\delta>0$ such that length of the meridian of any tube in $\double M_n$ is at most $D_\delta \cdot n$. This in turn implies that the radius of each such tube can be bounded by $E_\delta \log(n)$ for some constant $E_\delta>0$, depending on $\delta$ only. Combining this with our estimate on the diameter of the thick part and the estimates on volume from Proposition \ref{volume_bounds} this implies item (a).

We proceed to item (b). We observe that, for $\delta$ below the Margulis constant in $\HH^3$, the $\delta$-thin part of $M_n$ is simply connected. In particular, any closed geodesic that passes through the $(\delta/2)$-thin part of $M_n$ has length at least $\dist((M_n)_{\geq \delta}, (M_n)_{< \delta/2})$, which is uniformly bounded from below (by applying Lemmas \ref{small_cusps} and \ref{large_cusps}) to the $(\delta/2)$-thick part of $M_n$). The $(\delta/2)$-thick part of $M_n$ is bilipschitz to the $\widetilde{\delta}$-thick part of $Y_n$ for some uniform $\widetilde{\delta}>0$ (by Lemmas \ref{small_cusps} and \ref{large_cusps}). The systole of the latter is bounded from below by the distance between two distinct faces of $O_{\geq \widetilde{\delta}}$, which gives us a lower bound on the systole of the $(\delta/2)$-thick part of $M_n$. Together with the uniform bound on the length of geodesics that pass through the thin part, this implies a lower bound on $\sys(M_n)$.

For item (c) we use our combinatorial bounds again. Recall from the proof of Lemma \ref{large_cusps} that, with probability tending to $1$ as $n\to \infty$, the total cusp length $L$ satisfies $L^2 > n^{1/4-o(1)}$. \cite[Corollary 6.13]{Futer_Purcell_Schleimer} now immediatley implies the result.
\end{proof}

\subsection{Benjamini--Schramm convergence}

\subsubsection{Coxeter groups}

Let $T$ be ideal terahedron obtained by cutting $O$ along all of its median hyperplanes. Let $\Gamma_T$ be the associated reflection group. It is a Coxeter group with presentation
\[
\Gamma_T = \langle \sigma, \tau_1, \tau_2, \tau_3 | \tau_i^2, \sigma^2, [\tau_i, \tau_j], (\sigma\tau_j)^4 \rangle 
\]
This group is useful for us because of the following lemma.

\begin{lem} \label{orbifold_cover}
  Any $X \in \Theta_n$ is an orbifold cover of $T$. 
\end{lem}

\begin{proof}
  Let $X \in \Theta_n$ and $\mathcal G$ the graph dual to its tesselation by octahedra. Then $X$ is an orbifold cover of $O$ if and only if $\mathcal G$ is 4-edge-colourable, which is not always the case.

  However if we replace each vertex of $\mathcal G$ by a cube with four outgoing edges placed at pairwise non-adjacent vertices we get a graph $\mathcal G'$. We colour its edges as follows : all edges between cubes are coloured with $\sigma$, and inside the cube we choose the unique colouring corresponding to the labels (in $\ZZ/3\ZZ$)---this makes sense since the edge corresponds to a face of $O$ and the adjacent edges of $O$ are each specified by a $\tau_i$. This specifies a unique map $X \to T$ which is an orbifold cover. 
\end{proof}


\subsubsection{Invariant random subgroups}\label{sec_BSconv}

Let $G$ be a Lie group (we will only consider $G = \mathrm{PGL}_2(\CC)$). We recall from \cite{7s} that an {\em invariant random subgroup} of $G$ is a Borel probability measure on the Chabauty space of closed subgroups of $G$ (a compact Hausdorff topological space the definition of which can be found in loc. cit.) which is invariant under the action of $G$ on this space by conjugation. 

An important constructions of such is the following: if $\Lambda \le G$ is a subgroup whose normaliser is a lattice $\Gamma$ in $G$ then the closure of the conjugacy class of $\Lambda$ supports a unique invariant random subgroup (the image of Haar measure on $G/\Gamma$). We denote this by $\mu_{\Lambda}$.

Using this we can associate an invariant random subgroup to the random variable $M_n$ as follows: let $\Xi_n^{\mathrm{hyp}}$ be the subset of manifolds in $\Xi_n$ which support a complete hyperbolic structure with totally geodesic boundary. For $M \in \Xi_n^{\mathrm{hyp}}$ we consider the hyperbolic orbifold on $M$ whose singular locus is its boundary $\pl M$ (a mirror) and the hyperbolic structure on the interior is that of $M$. This is a compact hyperbolic orbifold and we choose an arbitrary monodromy group $\Gamma_M \le G$ for it and let 
\[\mu_M = \mu_{\Gamma_M}.\]
If $M \not\in \Xi_n^{\mathrm{hyp}}$ we take $\mu_M$ to be the Dirac mass at the trivial subgroup. We put:
\begin{equation} \label{IRS}
  \mu_n = \sum_{M \in \Xi_n} \mathbb P(M_n = M) \mu_M. 
\end{equation}

We also need to define some other invariant random subgroups which will play a role in what follows. Consider the ideal octahedron $O$ as a complete hyperbolic orbifold (all faces being mirrors). Let $\Gamma_O$ be its orbifold fundamental group, which is generated by the reflections on the sides of $O$. Let $Q$ be the group generated by the rotations of angle $2\pi/3$ in the faces. Then $P = Q \backslash O$ is an hyperbolic orbifold; let $\Gamma_P$ be its orbifold fundamental group, which we view as a lattice in $G$ (we need this larger group because not every manifold in $\Xi_n$ is an orbifold cover of $O$). 

Since $O$ is right-angled, mapping four reflections on nonadjacent faces to the identity gives a map 
\[
\pi : \Gamma_O \to \ast_{i=1}^4 \ZZ/2\ZZ
\]
(each remaining face maps to the generator of one of the free factors), and the latter is isomorphic to $D_\infty * D_\infty$ where $D_\infty = \ZZ/2\ZZ * \ZZ/2\ZZ$ is the infinite dihedral group. Let $O^\infty$ be the associated cover; it is the infinite hyperbolic polyhedron obtained by gluing copies of $O$ in a 4-valent tree pattern, along non-adjacent faces ($D_\infty*D_\infty$ acts via its action on the 4-valent tree). Note that if we view the mirrors as a totally geodesic boundary $O^\infty$ is the universal cover of any manifold in some $\Theta_n$. 

Since $Q$ respects the colouring of the faces of $O$ we have that $\ker(\pi)$ is a normal subgroup in $\Gamma_P$. We let $\mu_{O^\infty}$ be the invariant random subgroup of $G$ associated to the normal subgroup $\ker(\pi) \le \Gamma_P$. Our main result in this section is then the following. 

\begin{prop}
  The invariant random subgroup $\mu_n$ converges to $\mu_{O^\infty}$
\end{prop}

\begin{proof}
  Let $\mu_n'$ be the invariant random subgroup associated to the random variable $Y_n$. We will first prove that the sequence $\mu_n'$ converges to $\mu_{O^\infty}$. Let $\Gamma_P$ be the group defined above and $\Lambda = \ker(\pi) \le \Gamma_P$. Every $M \in \Xi_n$ admits a (possibly non-continuous) picewise isometric map to $O$ (by mapping its marked octahedra to $O$). The non-continuity comes from the rotations made when gluing faces so the composition $M \to O \to P$ is continuous and hence a covering map. Let $\Gamma_n \le \Gamma_P$ be the invariant random index-$12n$ subgroup corresponding to $M_n$.

  The Schreier graph of $\Gamma_P/\Gamma_n$ with loops removed is obtained from the graph dual to the tesselation of $M_n$ by replacing each vertex with a fixed graph $\mathcal Q$. The dual graphs follow the same distribution as the configuration model, and as this model of graphs BS-converges to the tree (this follows from \cite{Bol2}) we get that the random variable $\Gamma_n$ converges in distribution to $\Lambda$ (since the Schreier graph of $\Gamma_P/\Lambda$ is obtained from the tree by replacing vertices with $\mathcal Q$). Now $\mu_n$ is the IRS obtained by induction of $\Gamma_n$ from $\Gamma_0$ to $\mathrm{PGL}_2(\mathbb C)$, and $\mu_{O^\infty}$ by induction of $\Lambda$ (see \cite[11.1]{7S2} for the definition of induction). As induction is continuous we get that $\mu_n'$ converges to $\mu_{O^\infty}$. 

  Now if $\mu_n''$ is the IRS associated to $Z$ it follows immediately from the convergence of $\mu_n'$ together with Lemma \ref{small_cusps} that we also have $\lim(\mu_n'') = \mu_{O^\infty}$. 

  We pass to the larger space of random pointed metric spaces with Benjamini--Schramm topology (see \cite[Section 5]{Gelander_lecture}). If $\mu$ is an IRS of $\mathrm{PGL}_2(\CC)$ we denote by $\mu^{\ge\delta}$ the random pointed compact manifold with boundary which comes from conditioning the point to be in the thick part (note that doing so we lose all invariance properties). It follows from the previous paragraph that $(\mu_n'')^{\ge\delta}$ converges to $\mu_{O^\infty}^{\ge\delta}$ and from Lemma \ref{large_cusps} that $\mu_n^{\ge\delta}$ also does.

  Now the map $(M,x) \mapsto (M_{\ge\delta}, x)$ is an homeomorphism onto its image: it is continuous (immediate) and injective (the boundary of the thin part determines the complex length of the core geodesic if it is a tube, and the isometry class of the cross-section if it is a torus), and the space of hyperbolic manifolds pointed in their thick part is compact. We can thus conclude that $\mu_n$ converges to $\mu_{O^\infty}$.
\end{proof}


\bibliographystyle{plain}
\bibliography{bib}

\end{document}